\documentclass[11pt]{article}
\usepackage{mathptmx}
\usepackage[rmargin=0.35in,lmargin=0.35in,tmargin=0.35in,bmargin=0.5in,paperwidth=5.5in,paperheight=8.5in]{geometry}
\usepackage{amsmath, amssymb, amsthm}
\usepackage{graphicx}
\usepackage{mathrsfs}
\usepackage{float}

\usepackage{booktabs}
\usepackage{array}
\newcolumntype{M}[1]{>{\centering\arraybackslash}m{#1}}

\usepackage{subcaption}
\newtheorem{theorem}{Theorem}
\newtheorem{corollary}{Corollary}[theorem]
\newtheorem{lemma}{Lemma}
\theoremstyle{definition}

\theoremstyle{remark}

\usepackage[
    style=alphabetic
]{biblatex}
\DeclareMathOperator{\caps}{caps}
\DeclareMathOperator{\connpts}{connpts}

\title{\textbf{\MakeUppercase{Tile Numbers of Knot Corner Mosaics}}}
\author{\bf Ezra Jordan Aylaian\\
Duke University}
\date{\today}

\begin{document}

\maketitle

%\begin{abstract}
    %A knot mosaic is a grid of pictorial tiles representing a tame knot or link. Recently, two groups independently introduced a new set of tiles. We call mosaics made with these new tiles corner mosaics. The (corner) tile number is the minimum number of tiles needed to represent a knot or link as a (corner) mosaic. We show that the corner tile number lies strictly between the tile number and $3$ times the tile number, resolving a question of Heap et al. We also show that the only knots and links with corner tile number $<12$ and no unlinked, unknotted components are the Hopf link $P(1,1)$, the trefoil knot $3_1$, Solomon's knot $P(1,1,1,1)$, the connect sum of two Hopf links $P(1,1) \# P(1,1)$, the cinquefoil knot $5_1$, the star of David link $P(1,1,1,1,1,1)$, the figure-eight knot $4_1$, and the three-twist knot $5_2$.
%\end{abstract}

\section{Introduction}

In 2008, Lomonaco and Kauffman \cite{lomonacokauffman} introduced knot mosaics, a system for representing links as two dimensional grids of tiles. Lomonaco and Kauffman's tileset is shown in Figure \ref{fig:edgetiles}. In 2014, Kuriya and Shehab \cite{kuriyashehab} proved that knot mosaic theory is equivalent to tame knot theory.

\begin{figure}[h]
    \centering
    
    \begin{subfigure}{.075\textwidth}
        \centering
        \includegraphics[width=.7\textwidth]{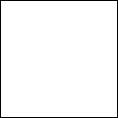}
        \caption{$T_0$}
    \end{subfigure}
    \begin{subfigure}{.075\textwidth}
        \centering
        \includegraphics[width=.7\textwidth]{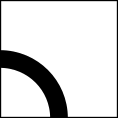}
        \caption{$T_1$}
    \end{subfigure}
    \begin{subfigure}{.075\textwidth}
        \centering
        \includegraphics[width=.7\textwidth]{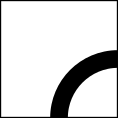}
        \caption{$T_2$}
    \end{subfigure}
    \begin{subfigure}{.075\textwidth}
        \centering
        \includegraphics[width=.7\textwidth]{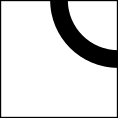}
        \caption{$T_3$}
    \end{subfigure}
    \begin{subfigure}{.075\textwidth}
        \centering
        \includegraphics[width=.7\textwidth]{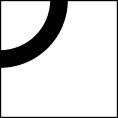}
        \caption{$T_4$}
    \end{subfigure}
    \begin{subfigure}{.075\textwidth}
        \centering
        \includegraphics[width=.7\textwidth]{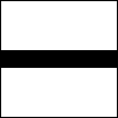}
        \caption{$T_5$}
    \end{subfigure}
    \begin{subfigure}{.075\textwidth}
        \centering
        \includegraphics[width=.7\textwidth]{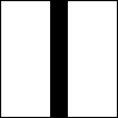}
        \caption{$T_6$}
    \end{subfigure}
    \begin{subfigure}{.075\textwidth}
        \centering
        \includegraphics[width=.7\textwidth]{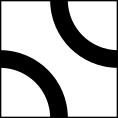}
        \caption{$T_7$}
    \end{subfigure}
    \begin{subfigure}{.075\textwidth}
        \centering
        \includegraphics[width=.7\textwidth]{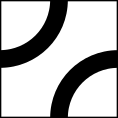}
        \caption{$T_8$}
    \end{subfigure}
    \begin{subfigure}{.075\textwidth}
        \centering
        \includegraphics[width=.7\textwidth]{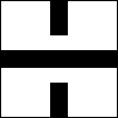}
        \caption{$T_9$}
    \end{subfigure}
    \begin{subfigure}{.075\textwidth}
        \centering
        \includegraphics[width=.7\textwidth]{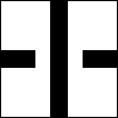}
        \caption{$T_{10}$}
    \end{subfigure}
    
    \caption{The edge tiles.}
    \label{fig:edgetiles}
\end{figure}

One natural invariant of a link defined using mosaics is the tile number, the minimum number of nonempty tiles needed to represent a given linkn as a mosaic. It was introduced in 2018 by Heap and Knowles \cite{heapknowles}. By 2019, Heap and Knowles \cite{heapknowles2} had enumerated all prime knots of tile number 24 or less and computed their tile numbers. Heap and LaCourt \cite{heaplacourt} did further work published in 2020.

However, Lomonaco and Kauffman's set of tiles is not the only natural set of tiles that can be used to build a mosaic theory. Aaron Heap et al. \cite{heap} and Eric Rawdon, Sayde Jude, and Lizzie Paterson \cite{judepaterson} independently concocted and studied the same alternative set of tiles. We call their new set of tiles the corner tiles and mosaics made with them corner mosaics. For clarity, we call Lomonaco and Kauffman's tiles the edge tiles and mosaics made with them edge mosaics. Figure \ref{fig:cornertiles} lists the corner tiles. Note that the pattern in each corner tile is just the pattern in the corresponding edge tile rotated clockwise $\pi/4$ radians.\footnote{To make this be true, our convention differs from \cite{heap} by a swap of $T_9$ and $T_{10}$.} We call $T_9$ and $T_{10}$ crossing tiles. Table \ref{tbl:examples} shows examples of corner mosaics.

\begin{figure}[h]
    \centering
    \begin{subfigure}{.075\textwidth}
        \centering
        \includegraphics[width=.7\textwidth]{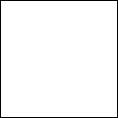}
        \caption{$T_0$}
    \end{subfigure}
    \begin{subfigure}{.075\textwidth}
        \centering
        \includegraphics[width=.7\textwidth]{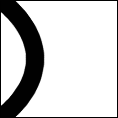}
        \caption{$T_1$}
    \end{subfigure}
    \begin{subfigure}{.075\textwidth}
        \centering
        \includegraphics[width=.7\textwidth]{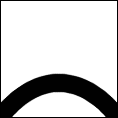}
        \caption{$T_2$}
    \end{subfigure}
    \begin{subfigure}{.075\textwidth}
        \centering
        \includegraphics[width=.7\textwidth]{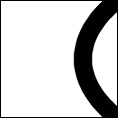}
        \caption{$T_3$}
    \end{subfigure}
    \begin{subfigure}{.075\textwidth}
        \centering
        \includegraphics[width=.7\textwidth]{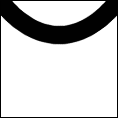}
        \caption{$T_4$}
    \end{subfigure}
    \begin{subfigure}{.075\textwidth}
        \centering
        \includegraphics[width=.7\textwidth]{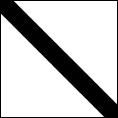}
        \caption{$T_5$}
    \end{subfigure}
    \begin{subfigure}{.075\textwidth}
        \centering
        \includegraphics[width=.7\textwidth]{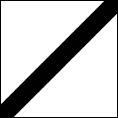}
        \caption{$T_6$}
    \end{subfigure}
    \begin{subfigure}{.075\textwidth}
        \centering
        \includegraphics[width=.7\textwidth]{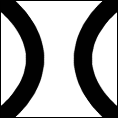}
        \caption{$T_7$}
    \end{subfigure}
    \begin{subfigure}{.075\textwidth}
        \centering
        \includegraphics[width=.7\textwidth]{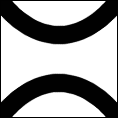}
        \caption{$T_8$}
    \end{subfigure}
    \begin{subfigure}{.075\textwidth}
        \centering
        \includegraphics[width=.7\textwidth]{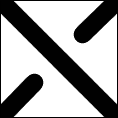}
        \caption{$T_9$}
    \end{subfigure}
    \begin{subfigure}{.075\textwidth}
        \centering
        \includegraphics[width=.7\textwidth]{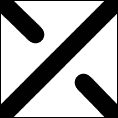}
        \caption{$T_{10}$}
    \end{subfigure}
    
    \caption{The corner tiles.}
    \label{fig:cornertiles}
\end{figure}

\bgroup
\def\arraystretch{3.5}
\begin{table}[h]
    \centering
    \begin{tabular}{cM{24mm}M{24mm}M{24mm}}
        & \includegraphics[width=.065\textwidth]{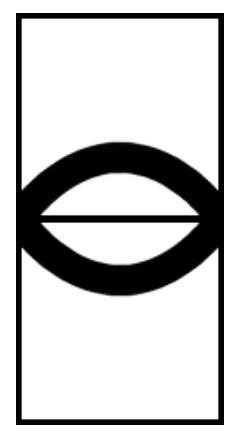}
        & \includegraphics[width=.17\textwidth]{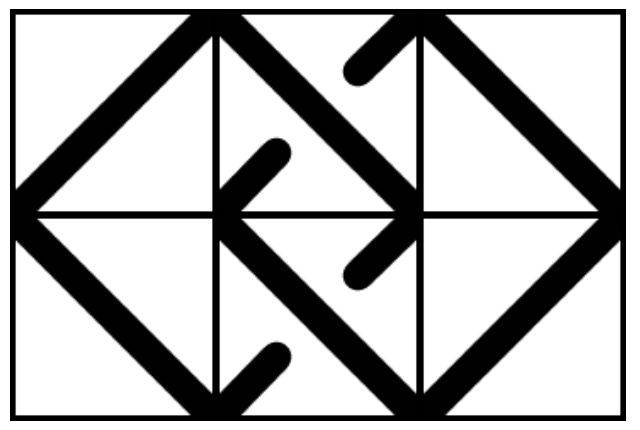}
        & \includegraphics[width=.17\textwidth]{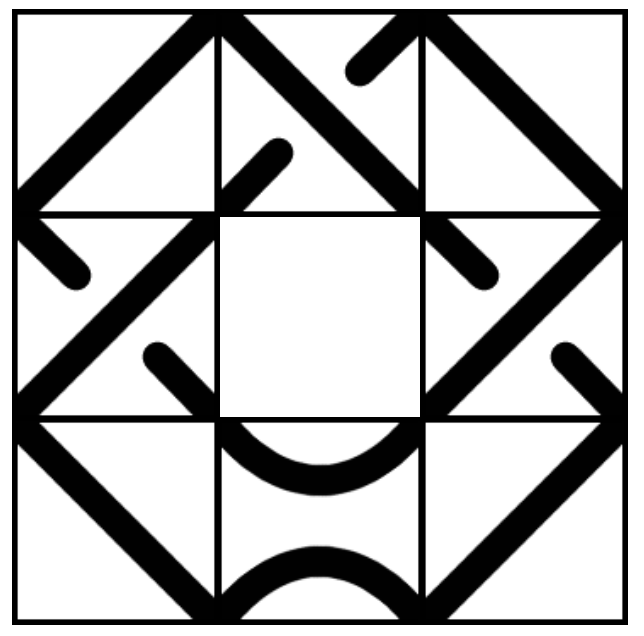} \\
        & \includegraphics[width=.17\textwidth]{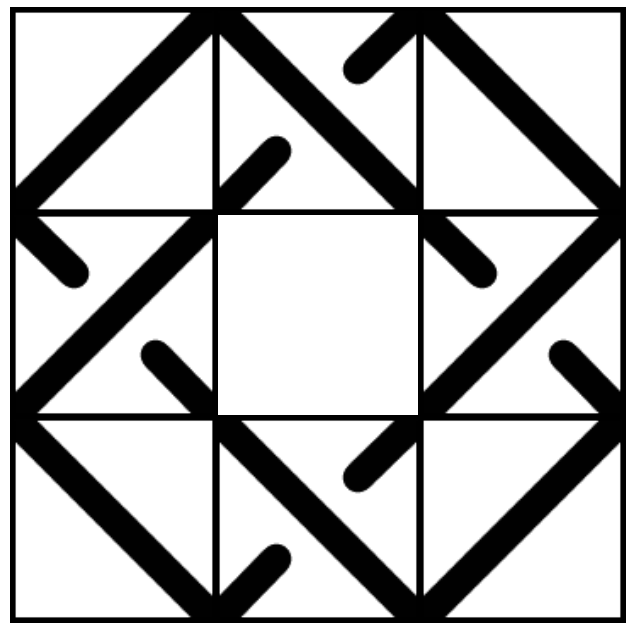}
        & \includegraphics[width=.17\textwidth]{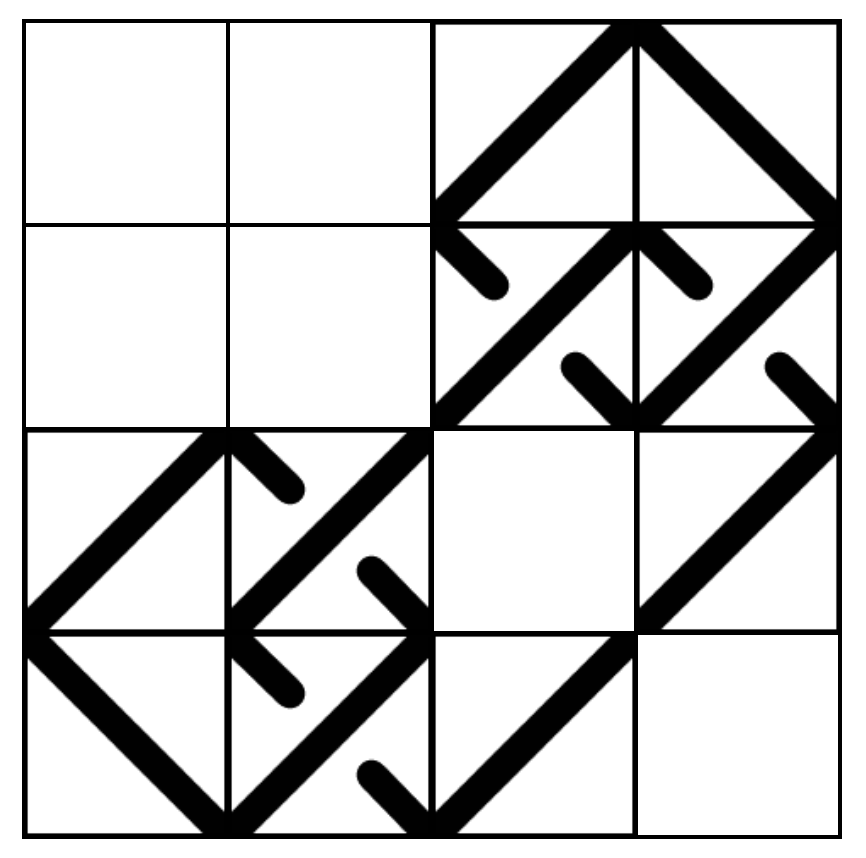}
        & \includegraphics[width=.17\textwidth]{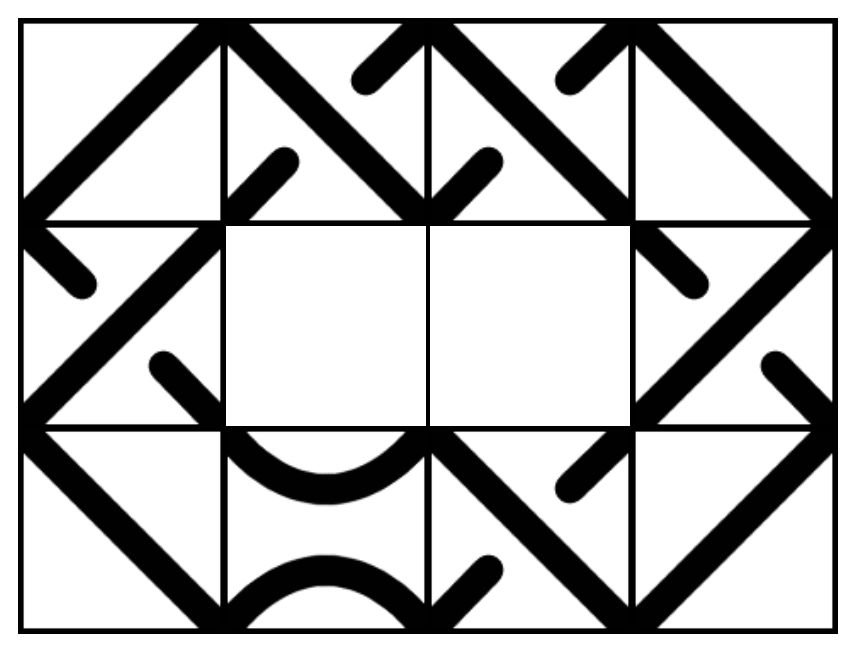} \\
        & \includegraphics[width=.17\textwidth]{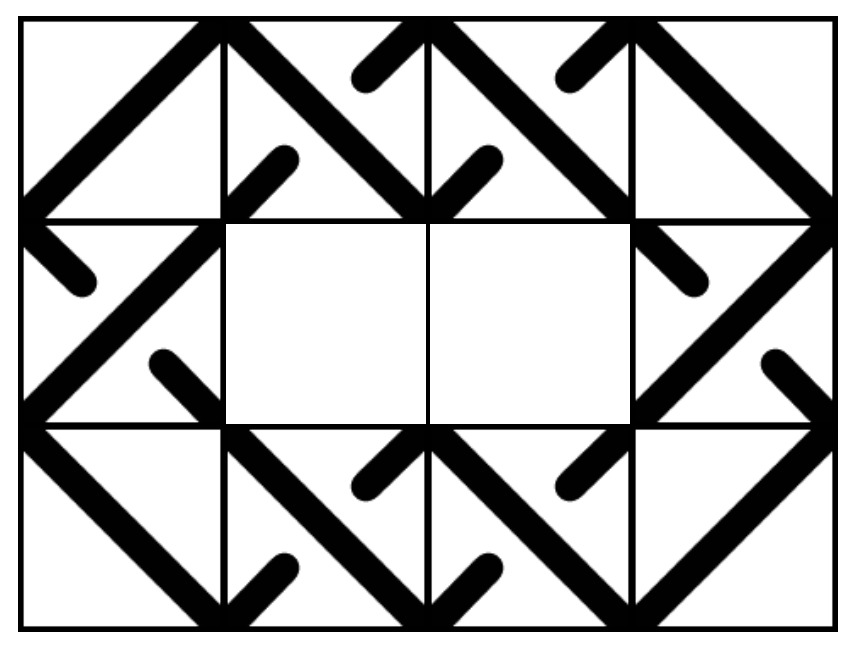}
        & \includegraphics[width=.17\textwidth]{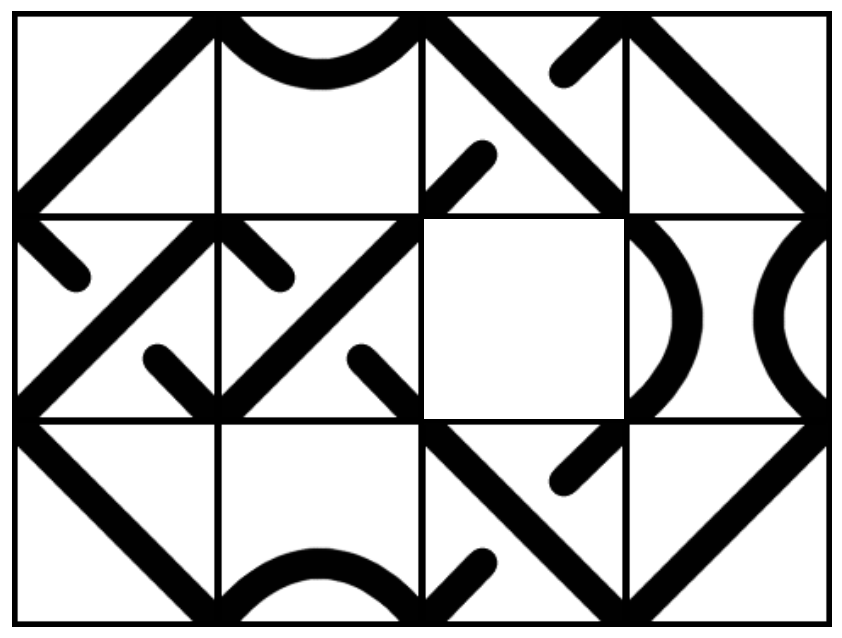}
        & \includegraphics[width=.17\textwidth]{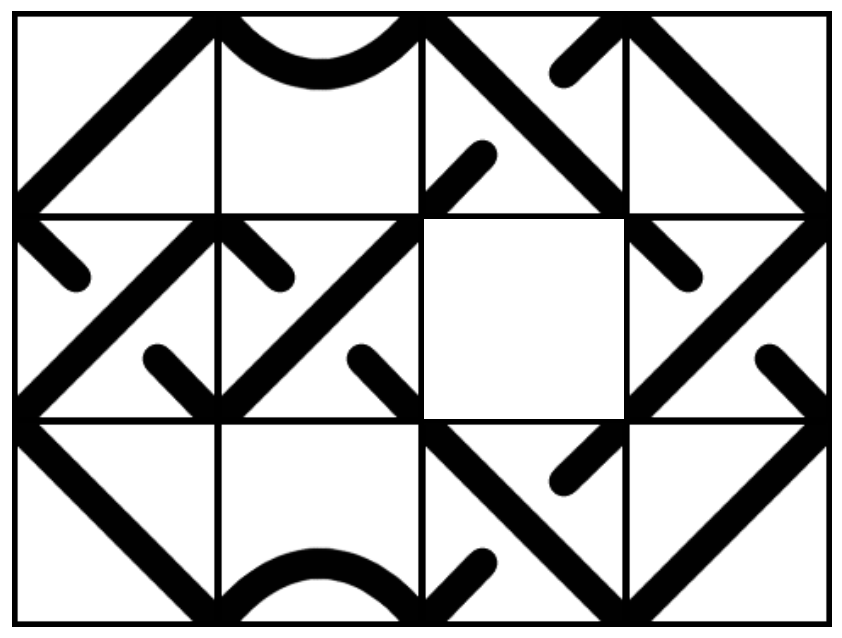}
    \end{tabular}
    \vspace{-.8em}
    \caption{Corner mosaics for, from left to right and top to bottom, the unknot, the Hopf link $P(1,1)$, the trefoil knot $3_1$, Solomon's knot $P(1,1,1,1)$, the connect sum of two Hopf links $P(1,1) \# P(1,1)$, the cinquefoil knot $5_1$, the star of David link $P(1,1,1,1,1,1)$, the figure-eight knot $4_1$, and the three-twist knot $5_2$.}
    \label{tbl:examples}
\end{table}
\egroup

Every mosaic system has an associated tile number, the minimum number of nonempty tiles needed to represent a link as a mosaic in that system. We denote the edge tile number of a link $L$ as $t(L)$ and the corner tile number as $t_C(L)$. The edge tile number is usually just called the tile number. Heap et al. \cite{heap} ask whether $t_C(K) \leq t(K)$ for all knots $K$. Theorem \ref{thm:main} answers this question affirmatively: for all links $L$,
\begin{equation*}
    t_C(L) + \caps(L) \leq t(L),
\end{equation*}
where $\caps(L)$ is a positive integer that is $\geq 4$ for $L$ with no unlinked, unknotted components. This shows that corner mosaics encode links more efficiently than edge mosaics, though it is not known whether the improvement grows asymptotically. On the other hand, Theorem \ref{thm:reverse} says that for any link $L$,
$$t_C(L) + \connpts(L) \geq t(L),$$
where $\connpts(L)$ is a positive integer that is bounded above by $2t_C(L) - 2$. This establishes
\begin{equation*}
    3t_C(L) - 2 \geq t(L).
\end{equation*}

Next, we want to classify links with small corner tile numbers. We reduce the problem to a combinatorial search which is carried out using a computer. The result is Theorem \ref{thm:classification}, which states that the only knots and links with corner tile number $<12$ and no unlinked, unknotted components are the Hopf link $P(1,1)$, the trefoil knot $3_1$, Solomon's knot $P(1,1,1,1)$, the connect sum of two Hopf links $P(1,1) \# P(1,1)$, the cinquefoil knot $5_1$, the star of David link $P(1,1,1,1,1,1)$, the figure-eight knot $4_1$, and the three-twist knot $5_2$.

\section{Inequalities Relating Tile and Corner Tile Numbers}

A link $L$ is said to have no unlinked, unknotted components if it does not have an unknotted component which is unlinked from the rest of $L$. For a link $L$, define the set of edge mosaics representing the tile number of $L$ as
\begin{equation*}
    \mathscr M^L = \{ M \mid M \text{ is an edge mosaic for } L \text{ with } t(L) \text{ nonempty tiles} \}.
\end{equation*} 
The set $\mathscr M_C^L$ of corner mosaics representing the corner tile number of $L$ is defined similarly.

\subsection*{Corner Mosaics Are More Efficient Than Edge Mosaics}

In the context of an edge mosaic, a cap is two edge-adjacent tiles that have the form of Figure \ref{fig:cap} up to rotation. Define
\begin{equation*}
    \caps(L) = \max \{ \text{number of caps in } M \mid M \in \mathscr M^L \},
\end{equation*}
where the counting is done so that any one tile may contribute to at most one cap. For example, Figure \ref{fig:unknotedge} shows that $\caps(\text{unknot}) = 2$. If $L$ has no unlinked, unknotted components, caps cannot overlap in a mosaic in $\mathscr M^L$, since if there are two overlapping caps in three tiles of a 2x2 submosaic, no matter which of the 11 corner tiles goes in the remaining slot, the number of nonempty tiles can be reduced.

\begin{figure}[ht]
    \centering
    \includegraphics[width=.8in]{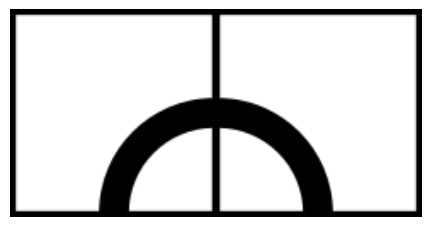}
    \caption{A cap.}
    \label{fig:cap}
\end{figure}

\begin{figure}[ht]
    \centering
    \includegraphics[width=.8in]{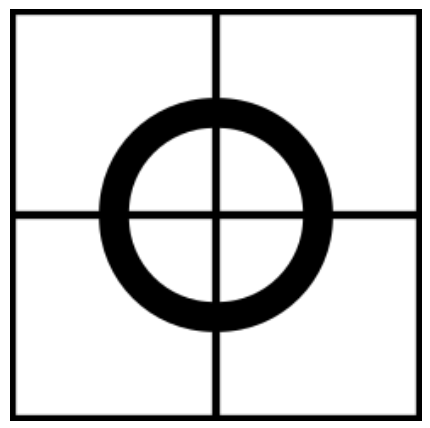}
    \caption{$\caps(\text{unknot}) = 2$.}
    \label{fig:unknotedge}
\end{figure}

\begin{theorem} \label{thm:main}
    For any link $L$, $t_C(L) + \caps(L) \leq t(L)$.
\end{theorem}

\begin{proof}
    Choose an edge mosaic for $L$ with $t(L)$ nonempty tiles and $\caps(L)$ caps. We describe an algorithm which is illustrated in Table \ref{tbl:proof} from left to right. Rotate the mosaic $\pi/4$ radians clockwise and superimpose a secondary grid. Then remove the original grid and keep the secondary grid. This gives us a corner mosaic with nonempty tiles in a checkerboard pattern. Each edge tile maps to the corresponding corner tile in the checkerboard. However, for every cap in the mosaic we started with, we end up with two tiles which can be ``pushed in'' to a previously empty tile. Pushing in all the caps reduces the number of nonempty tiles by the number of caps. This gives a corner mosaic for $L$ with $t(L) - \caps(L)$ nonempty tiles.
\end{proof}

\begin{table}[h]
    \centering
    \begin{tabular}{cM{16mm}M{21.6mm}M{21.6mm}M{19.2mm}M{12.8mm}}
        & \includegraphics[width=.13\textwidth]{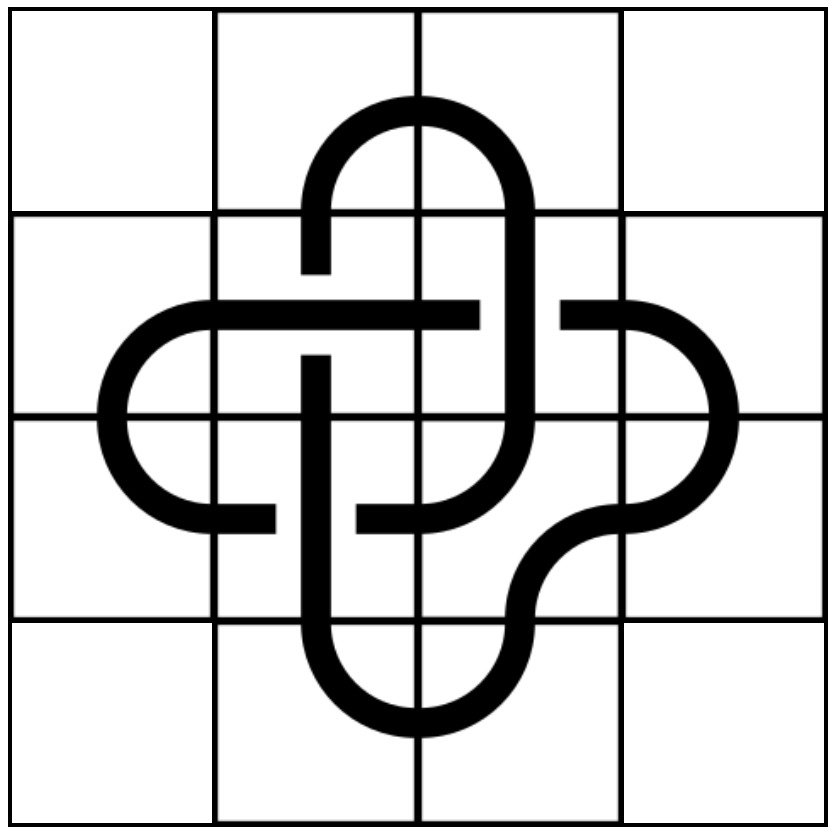}
        & \includegraphics[width=.17\textwidth]{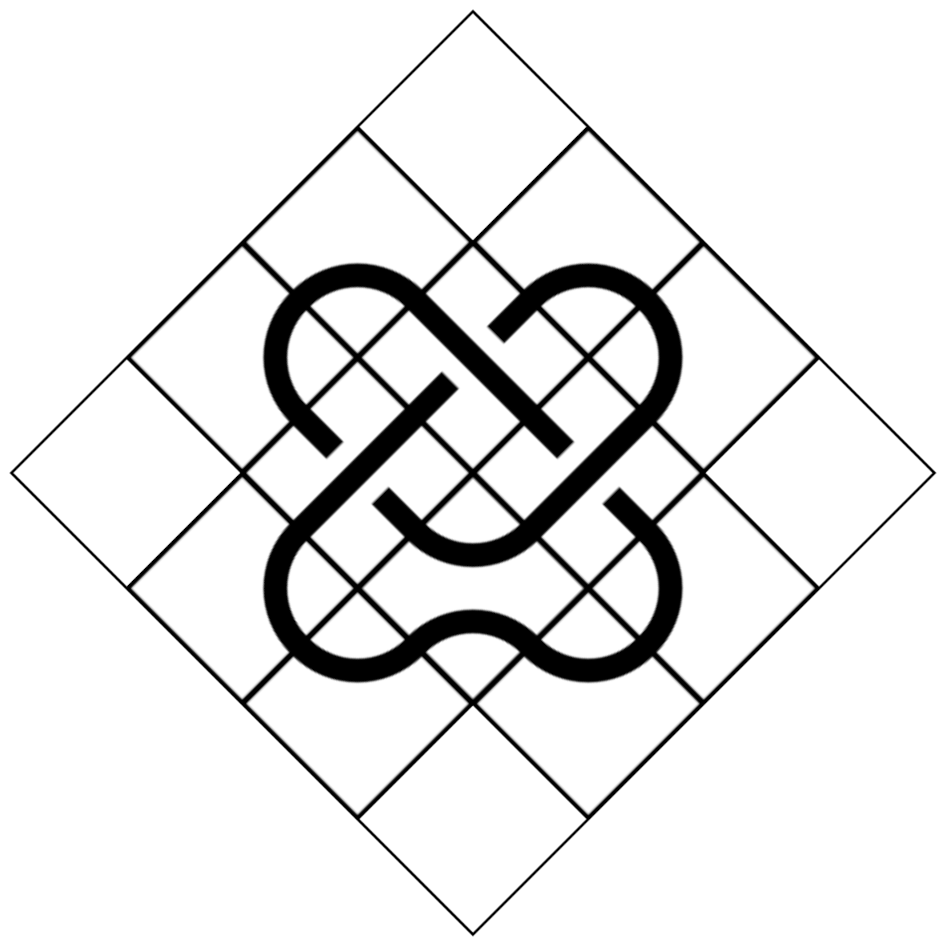}
        & \includegraphics[width=.17\textwidth]{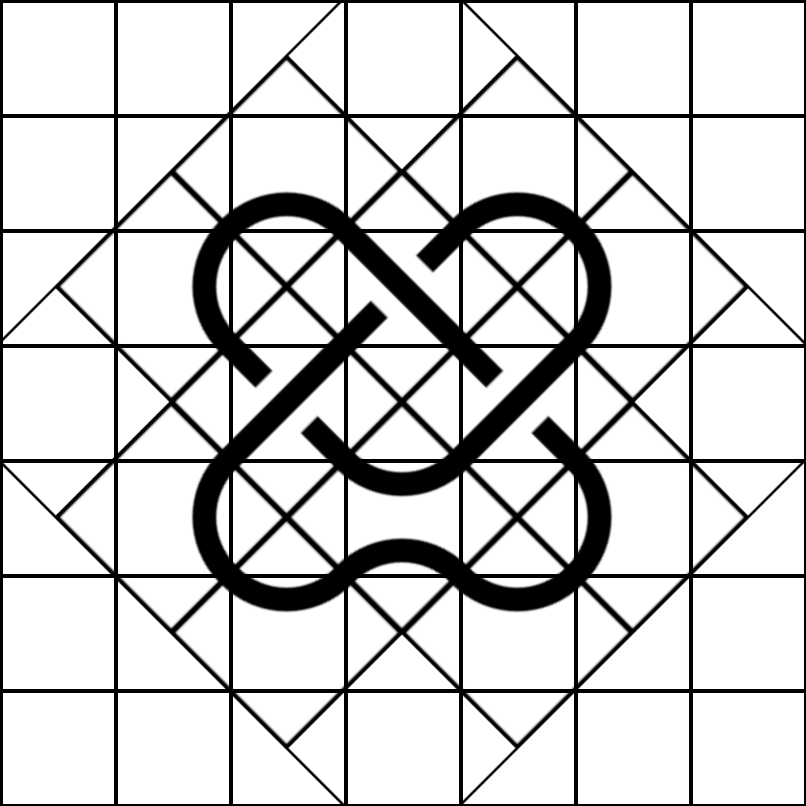}
        & \includegraphics[width=.15\textwidth]{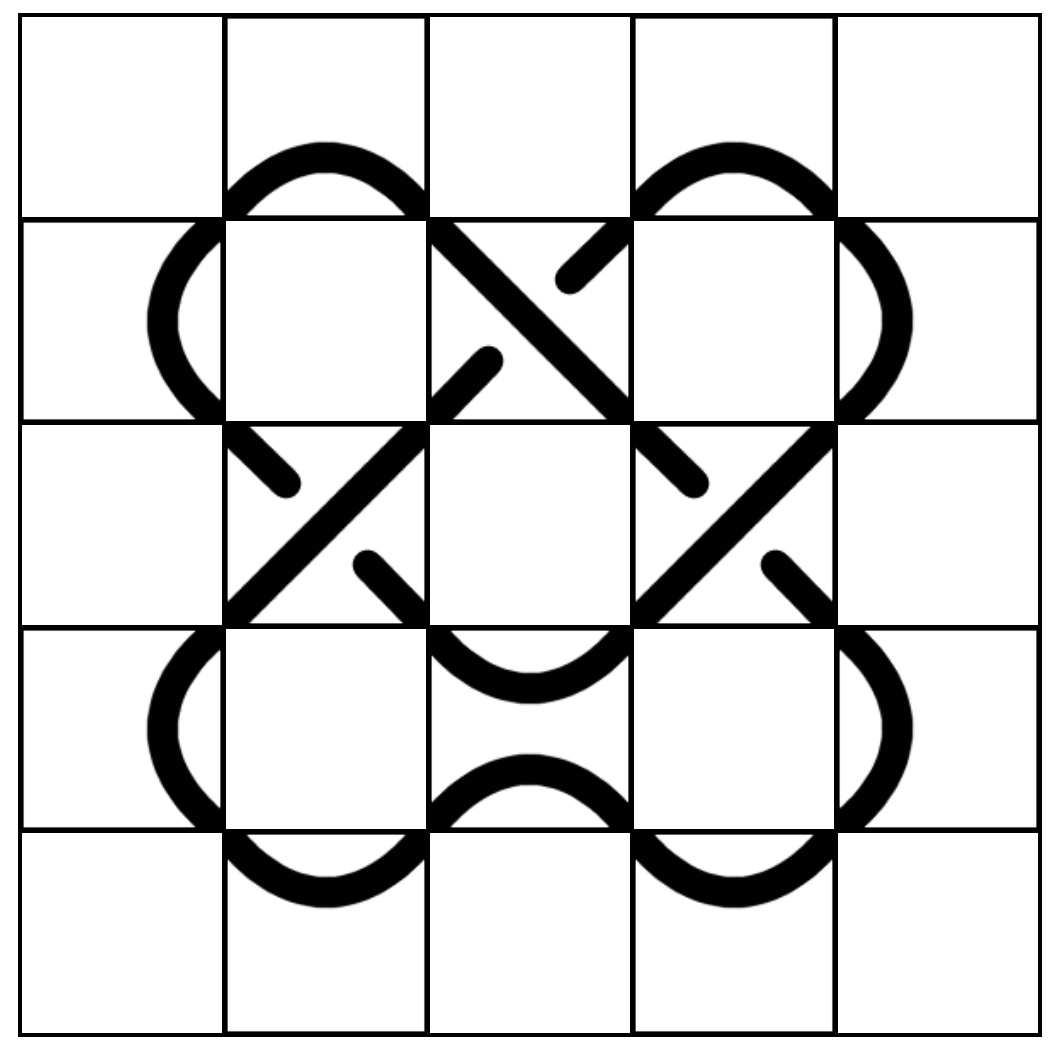}
        & \includegraphics[width=.10\textwidth]{Mosaics/RHT}
    \end{tabular}
    \caption{Rotate, replace grid, and push in caps.}
    \label{tbl:proof}
\end{table}

Define a checkerboard corner mosaic as a corner mosaic where every other square in a checkerboard pattern is empty. The proof of Theorem \ref{thm:main} establishes a tile-number-preserving bijection between edge mosaics and checkerboard corner mosaics.

\begin{corollary}
    If $L$ has no unlinked, unknotted components, $t_C(L) + 4 \leq t(L)$.
\end{corollary}
\begin{proof}
    By Lemma 6 of \cite{heapknowles}, for any link $L$, there is an edge mosaic in $\mathscr M^L$ such that the topmost nonempty row, bottommost nonempty row, leftmost nonempty column, and rightmost nonempty column contain only caps and empty tiles. If $L$ has no unlinked, unknotted components, the caps in these rows and columns don't overlap, so $\caps(L) \geq 4$.
\end{proof}

If $L$ is the unknot, then $t_C(\text{unknot}) = 2$, $t(\text{unknot}) = 4$, and $\caps(\text{unknot}) = 2$. An interesting question we do not answer is whether there exist $A,B$ with $A > 1$ and $At_C(L) + B \leq t(L)$.

\subsection*{Inequalities in the Other Direction}

The algorithm in the proof of Theorem \ref{thm:main} does not always take a representation of the edge tile number to a representation of the corner tile number. In these cases, $t_C(L) + \caps(L) < t(L)$. For example, $t(\text{Hopf link}) = 12$ and $\caps(\text{Hopf link}) = 4$, but $t_C(\text{Hopf link}) = 6$. Observe further that if we start with the corner mosaic for the Hopf link from Table \ref{tbl:examples}---which is a representation of the corner tile number---and ``push out'' the $T_5$ and $T_6$ tiles, we do not get a checkerboard mosaic. Hence it is just the last step of the proof that is not reversible and obstructs equality. Thinking about how to turn a corner mosaic into a checkerboard corner mosaic naturally leads us to the proof of the next theorem.

A connection point is a point on the boundary of a corner mosaic tile that is the endpoint of a curve drawn on the tile. For example, the corner mosaic for the Hopf link in Table \ref{tbl:examples} has 8 connection points. Define
\begin{equation*}
    \connpts(L) = \min \{ \text{number of connection points in } M \mid M \in \mathscr M_C^L \}.
\end{equation*}

\begin{theorem} \label{thm:reverse}
    For any link $L$, $t_C(L) + \connpts(L) \geq t(L)$.
\end{theorem}
\begin{proof}
    Choose a corner mosaic for $L$ with $t_C(L)$ nonempty tiles. Add in extra spacing and fill in extra tiles at each connection point of the original corner mosaic, as illustrated in Table \ref{tbl:proof2} from left to right. The result is a checkerboard corner mosaic with $t_C(L) + \connpts(L)$ nonempty tiles. Going through the tile-number-preserving bijection between edge mosaics and checkerboard corner mosaics gives an edge mosaic for $L$ with $t_C(L) + \connpts(L)$ nonempty tiles.
\end{proof}

\begin{table}[h]
    \centering
    \begin{tabular}{cM{20mm}M{40mm}M{40mm}}
        & \includegraphics[width=.12\textwidth]{Mosaics/RHT}
        & \includegraphics[width=.24\textwidth]{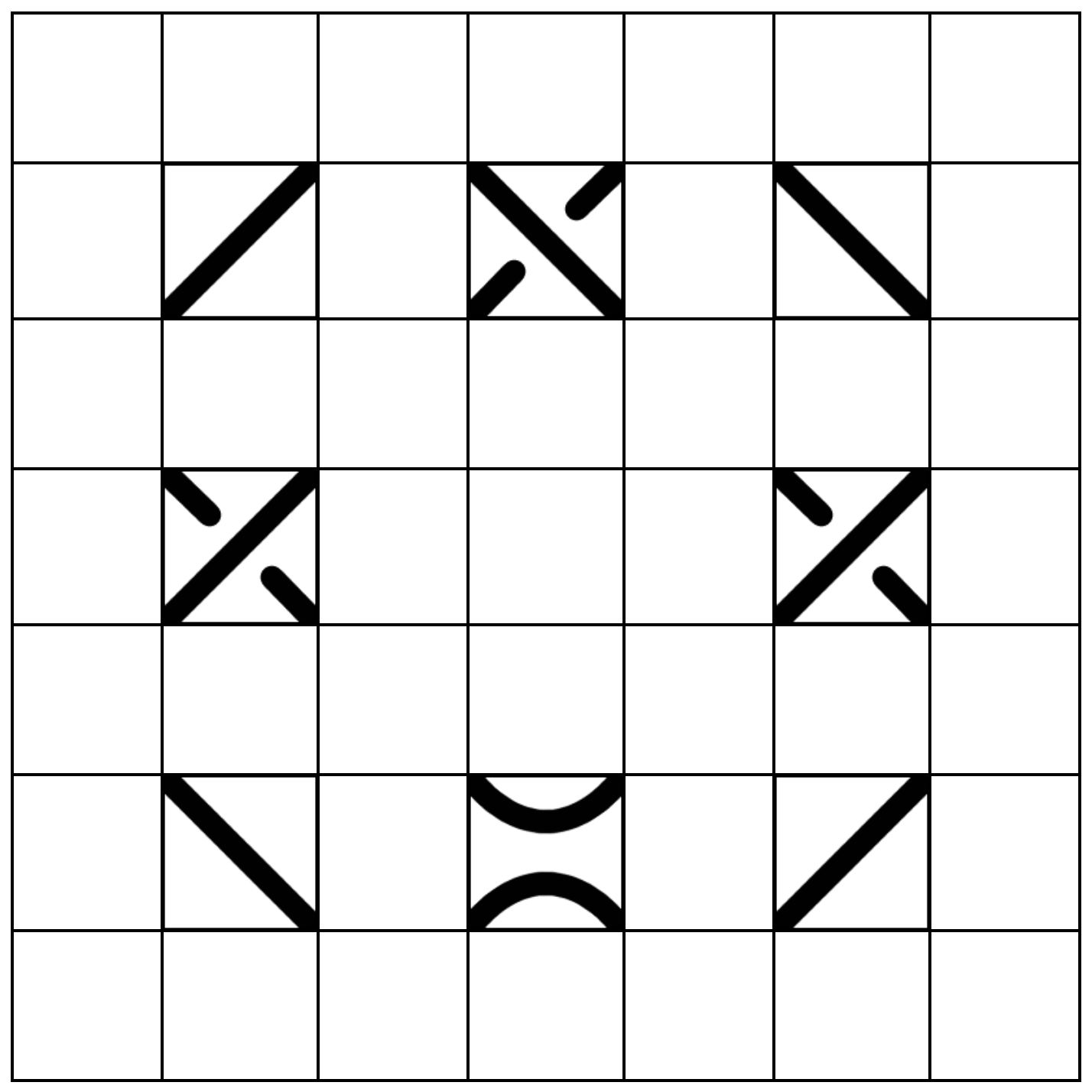}
        & \includegraphics[width=.24\textwidth]{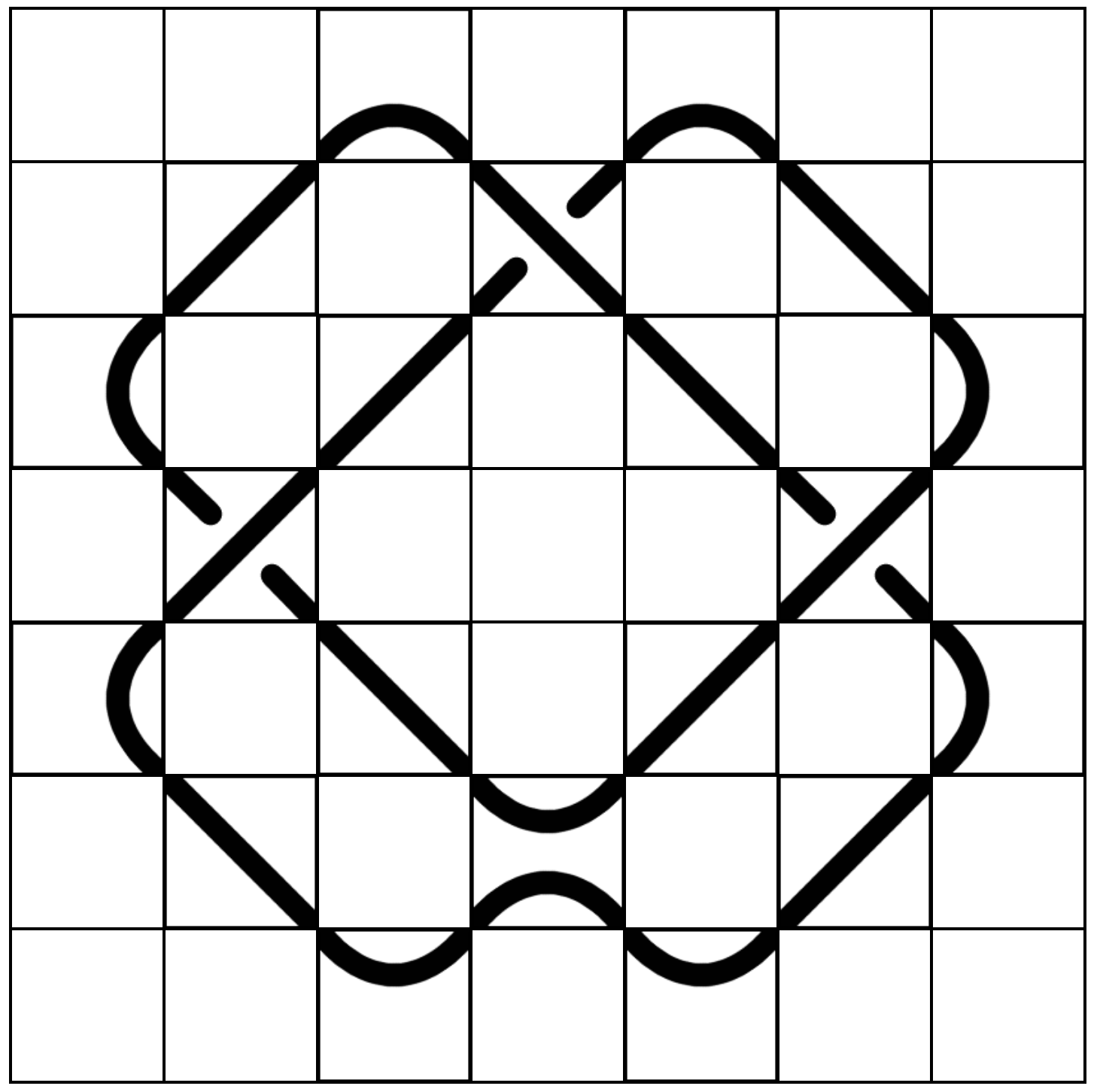}
    \end{tabular}
    \caption{Add extra spacing and fill in the missing connections.}
    \label{tbl:proof2}
\end{table}

\begin{corollary}
    For any link $L$, $3t_C(L) - 2 \geq t(L)$.
\end{corollary}
\begin{proof}
    We bound $\connpts(L)$. Mark the inside of each corner of each nonempty tile of the corner mosaic with 4 slots, so that there are exactly $4t_C(L)$ slots total. Each connection point takes up 2 slots. There are also 4 distinct slots that cannot be used by a connection point:
    \begin{enumerate}
        \item the bottom-left corner of the leftmost nonempty column's bottommost nonempty tile,
        % non-empty is hyphenated for page wrap
        \item the upper-left corner of the leftmost nonempty column's uppermost non-empty tile,
        \item the bottom-right corner of the rightmost nonempty column's bottommost nonempty tile,
        \item the upper-right corner of the rightmost nonempty column's uppermost nonempty tile.
    \end{enumerate}
    Hence there are at most $2t_C(L) - 2$ connection points.
\end{proof}

An interesting question we do not answer is to find the infimum of $A$ such that there exists a $B$ with $At_C(L) + B \geq t(L)$. Our results show that the infimum is $\geq 1$ and $\leq 3$.

\section{Links With Small Corner Tile Numbers}

The main theorem of this section classifies links with small corner tile numbers. We use figures of knot corner mosaics where the tiles are not filled in. In these figures, dark gray represents nonempty tiles, light gray represents possibly nonempty tiles, and white represents empty tiles.

\subsection*{Two Lemmas}

The following lemma, which describes what elements of $\mathscr M_C^L$ do and don't look like locally, is invaluable and will be used many times in the course of the classification.

\begin{lemma} \label{lmm:submosaics}
    Let $L$ have no unlinked, unknotted components and let $M \in \mathscr M_C^L$. Then mod rotation and reflection,
    \begin{enumerate}
        \item $M$ does not contain any of the subarrays depicted in Figure \ref{fig:canthappen},
        \item If $M$ has an instance of $T_2$, the tile below it is empty,
        \item If the subarray depicted in Figure \ref{fig:middleist5} occurs in $M$, the middle tile is $T_5$.
    \end{enumerate}
    Additionally, there exists a corner mosaic in $\mathscr M_C^L$ such that the subarrays in Figure \ref{fig:neednothappen} don't occur mod rotation and reflection.
    \begin{figure}[ht]
        \centering

        \begin{subfigure}{.18\textwidth}
            \centering
            \includegraphics[width=.8\textwidth]{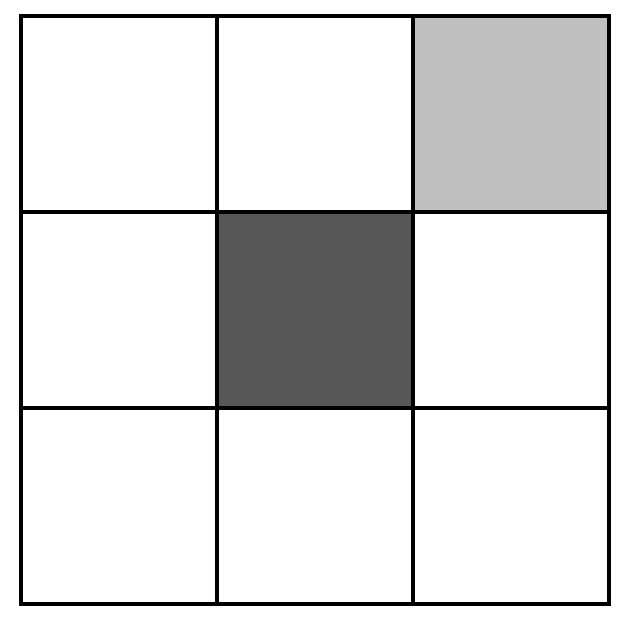}
        \end{subfigure}
        \begin{subfigure}{.18\textwidth}
            \centering
            \includegraphics[width=.8\textwidth]{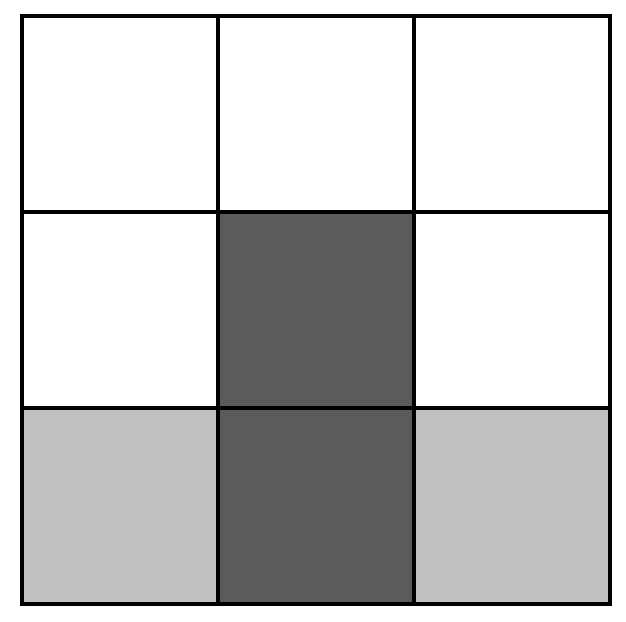}
        \end{subfigure}
        \begin{subfigure}{.18\textwidth}
            \centering
            \includegraphics[width=.8\textwidth]{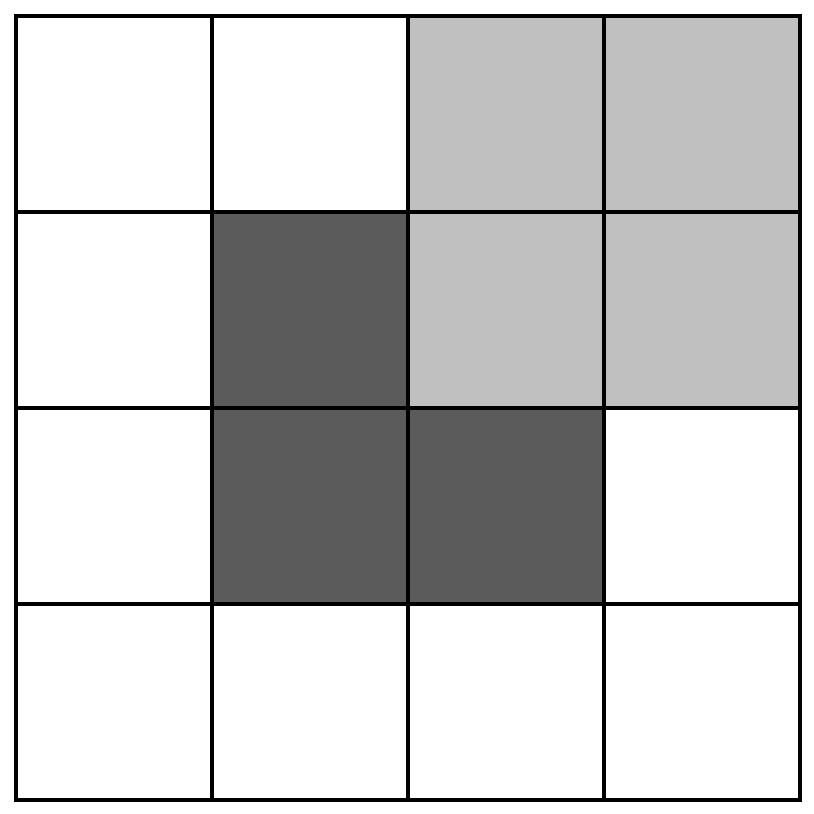}
        \end{subfigure}
        \begin{subfigure}{.21\textwidth}
            \centering
            \includegraphics[width=.9\textwidth]{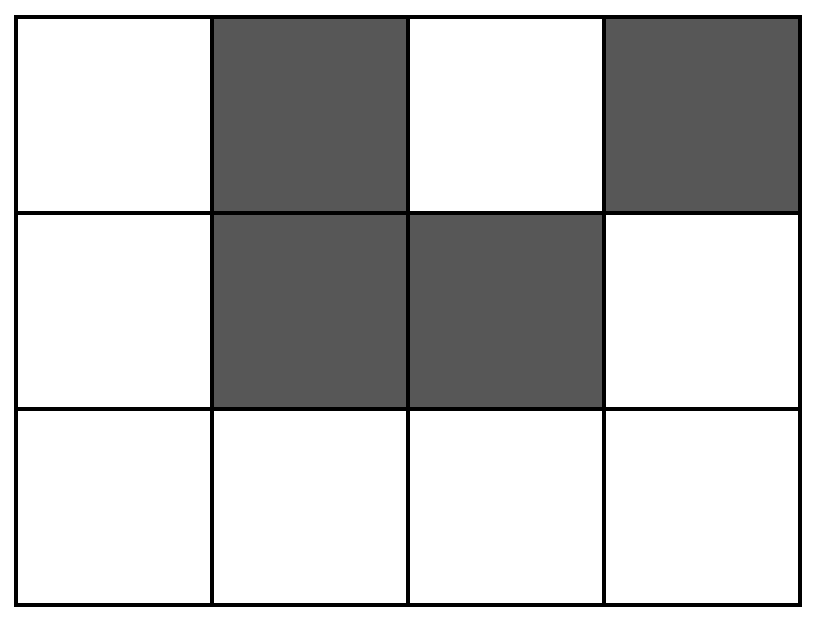}
        \end{subfigure}
        
        \caption{Subarrays that cannot occur in $M$.}
        \label{fig:canthappen}
    \end{figure}
    \begin{figure}[ht]
        \centering
        \includegraphics[width=.8in]{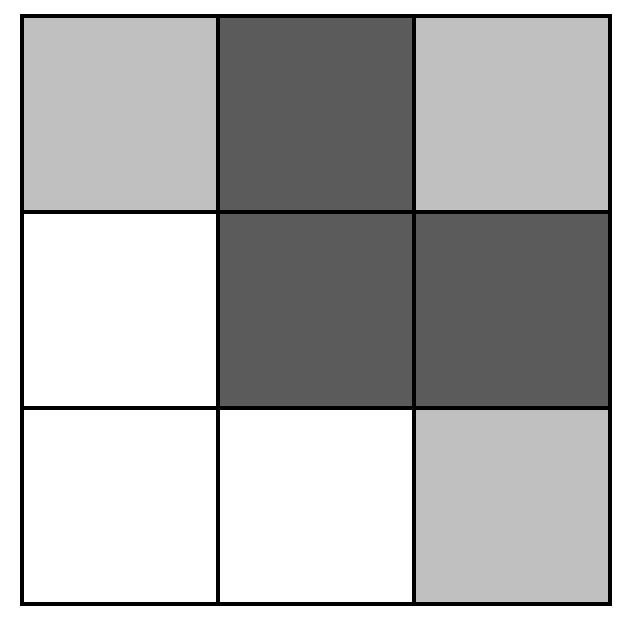}
        \caption{In this subarray, the middle tile must be $T_5$.}
        \label{fig:middleist5}
    \end{figure}
    \begin{figure}[H]
        \centering
        
        \begin{subfigure}{.18\textwidth}
            \centering
            \includegraphics[width=.8\textwidth]{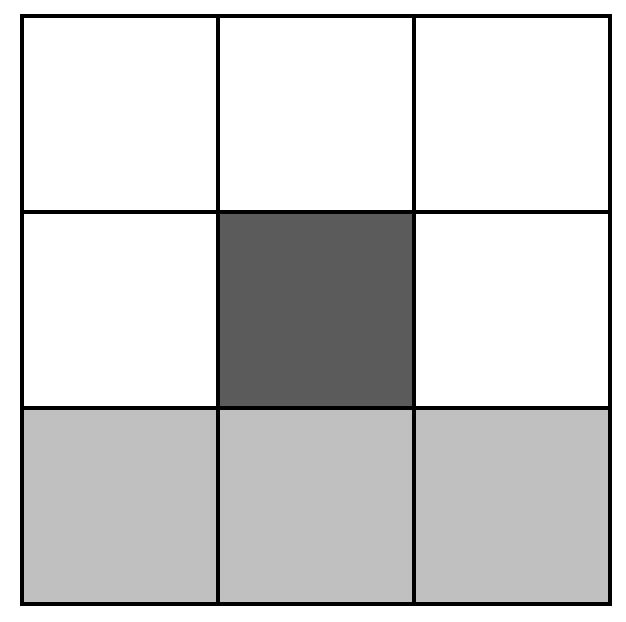}
        \end{subfigure}
        \hspace{.05em}
        \begin{subfigure}{.21\textwidth}
            \centering
            \includegraphics[width=.9\textwidth]{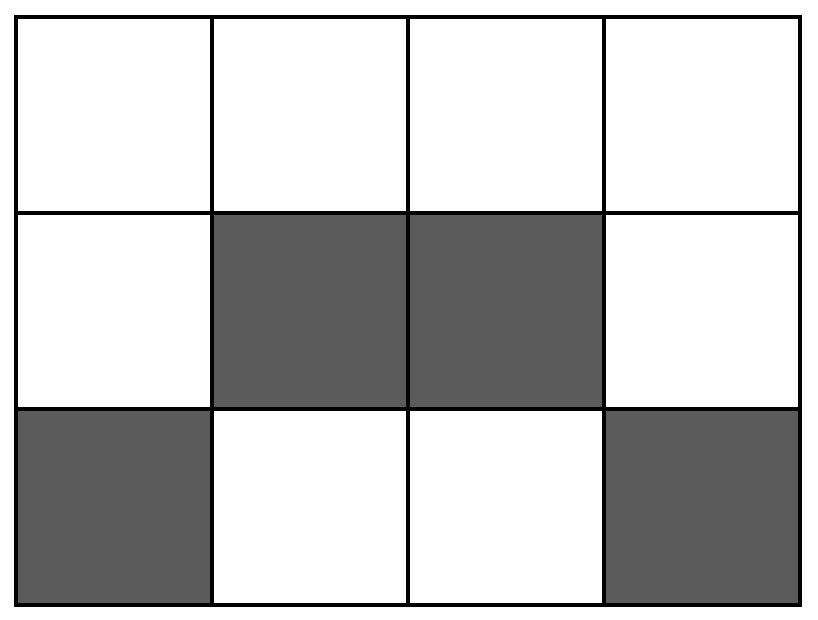}
        \end{subfigure}
        
        \caption{There exists a corner mosaic in $M_C^L$ that doesn't contain these subarrays.}
        \label{fig:neednothappen}
    \end{figure}
\end{lemma}

\begin{proof}
    None of the 10 nonempty tiles can be in the middle tile of Figure \ref{fig:canthappen} subarray 1, so that subarray cannot occur.
    
    % non-empty is hyphenated for page wrap
    In Figure \ref{fig:canthappen} subarray 2, the middle tile can only be $T_2$. Consider any non-empty tile that could be below $T_2$. In every case, either the subarray can be reduced in tiles while representing the same link or the subarray contains an unlinked unknot. This argument also shows that every instance of $T_2$ has an empty tile below it.
    
    In the subarray of Figure \ref{fig:middleist5}, the middle tile can only be $T_3$, $T_4$, or $T_5$. But the tile to the right of a $T_3$ is empty and the tile above a $T_4$ is empty, so the only option left is $T_5$.
    
    In Figure \ref{fig:canthappen} subarray 3, the bottom-left-most nonempty tile can only be $T_5$, whence the tile above it is either $T_2$ or $T_6$. If it were $T_2$, then we could reduce the subarray in tiles while representing the same link, so it must be $T_6$. The same argument shows that the tile to the right of the bottom-left-most nonempty tile tile is $T_6$. At this point, the subarray looks like the left of Figure \ref{fig:lemmaproofineff}. However, we can now reduce it in tiles while representing the same link by using the subarray depicted on the right of Figure \ref{fig:lemmaproofineff} instead. Hence the original subarray cannot occur.
    \begin{figure}[h]
        \centering

        \begin{subfigure}{.21\textwidth}
            \centering
            \includegraphics[width=.9\textwidth]{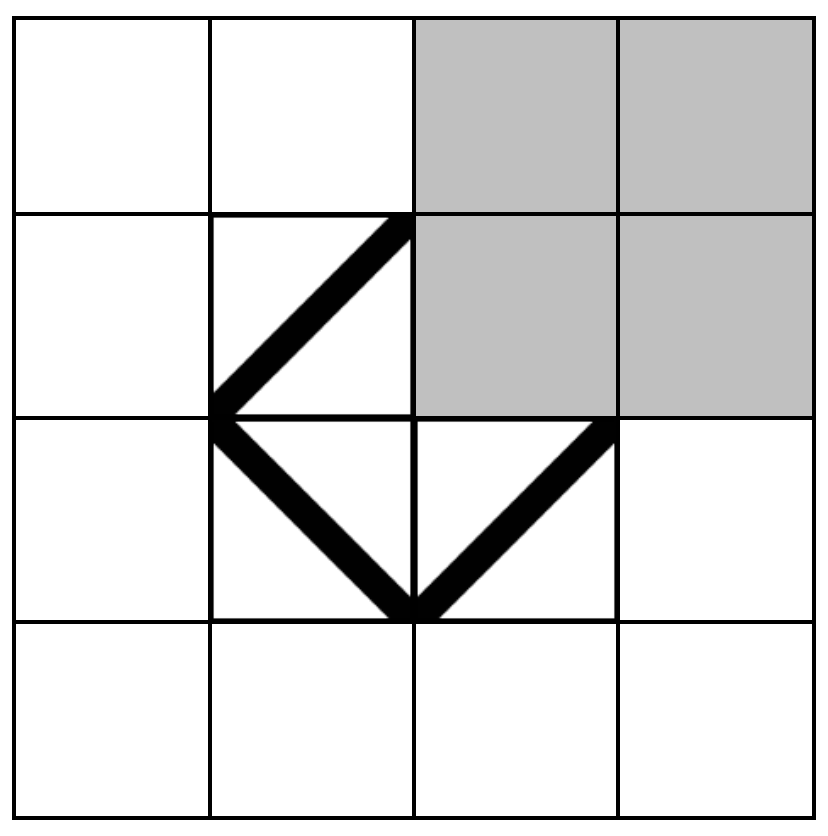}
        \end{subfigure}
        \hspace{.2em}
        \begin{subfigure}{.21\textwidth}
            \centering
            \includegraphics[width=.9\textwidth]{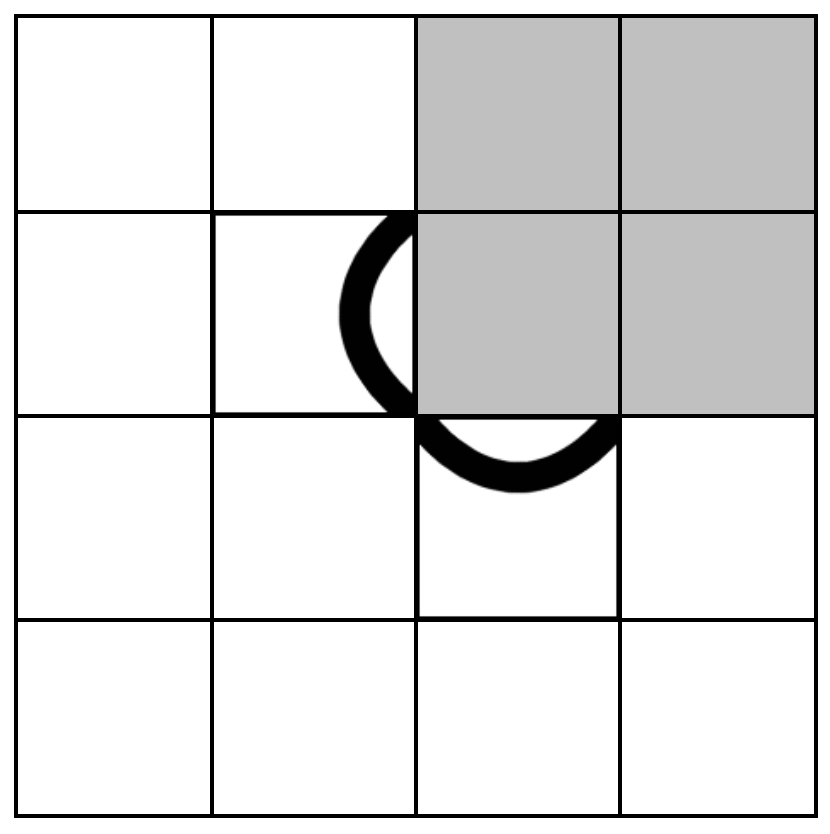}
        \end{subfigure}
        
        \caption{Figure \ref{fig:canthappen} subarray 4 must take the form on the left, which can be replaced with the subarray on the right without changing which link is being represented.}
        \label{fig:lemmaproofineff}
    \end{figure}

    In Figure \ref{fig:canthappen} subarray 4, the bottom-left-most nonempty tile is $T_5$, whence the tile to the right of it can only be $T_6$, from which we see that the tile above the $T_5$ can only be $T_1$ or $T_6$. But now this subarray can be reduced in tiles while representing the same link.
    
    For the last sentence, start by picking any $M \in \mathscr M_C^L$. When Figure \ref{fig:neednothappen} subarray 1 occurs in $M$, we know that the bottom-middle tile is empty by Figure \ref{fig:canthappen} subarray 2. Further, the nonempty tile in the middle row can can only be $T_2$, so we can push it down to a $T_4$ in the bottom-middle tile, making the middle tile empty. Applying this process to all occurrences of the subarray in $M$ mod rotation and reflection yields a new mosaic $M'$ that is still in $\mathscr M_C^L$ but has no instances of Figure \ref{fig:neednothappen} subarray 1. Now, in each instance of Figure \ref{fig:neednothappen} subarray 2, the tiles in the middle row are either $T_5$ and $T_6$ respectively or they are both $T_2$. In either case, both tiles can be pushed down to two $T_4$'s in the bottom row. Applying this to all occurrences of the subarray in $M'$ mod rotation and reflection yields $M'' \in \mathscr M_C^L$ with no instances of either subarray of Figure \ref{fig:neednothappen}.
\end{proof}

The proof of the classification uses polyominoes. For our purpose, a polyomino with no adjectives is a choice of finitely many tiles in the plane, i.e. a finite subset of $\mathbb Z^2$, mod rotation, reflection, and translation. This is unusual; polyominoes are usually required to be edge-connected, which means that the graph with a vertex for each tile and an edge for each pair of tiles sharing an edge is connected. This condition will not be useful for us. Instead, we say a polyomino is corner-connected if the graph with a vertex for each tile and an edge for each pair of tiles sharing an edge \emph{or corner} is connected.

Given a corner mosaic, one can forget all information other than which tiles are nonempty, yielding a (not necessarily corner-connected) polyomino. We call this process polyominoification. A corner mosaic is called corner-connected if its polyominoification is. A corner-connected polyomino is said to be in compliance with Lemma \ref{lmm:submosaics} if it does not contain any subarrays that Lemma \ref{lmm:submosaics} either outlawed or said need not occur. The key idea in the proof of the classification theorem is that in order to enumerate all $n$-tile corner mosaics, it is helpful to compute the image of
$$\{ M \in \mathscr M_C^L \text{ for some unknot free } L \text{ with } t_C(L) = n \}$$
under polyominoification. We will not compute the image exactly, but we will compute a larger set: the corner-connected $n$-tile polyominoes in compliance with Lemma \ref{lmm:submosaics}. This is the focus of the next Lemma.

A domino is a 2-tile polyomino. There is a unique edge-connected domino, but there are two corner-connected dominoes.

\begin{lemma} \label{lmm:algorithm}
    Start with a singleton $\mathscr P_3$ containing the unique edge-connected domino. Repeat the following as $i$ ranges from $3$ to $n$: for each $P \in \mathscr P_{i-1}$, for each tile $T$ in $P$, for each of the 8 tiles $U$ in the plane sharing an edge or corner with $T$, if $U$ is not already in $P$, add it to $P$ to get a corner-connected polyomino $P'$. If $\mathscr P_i$ does not yet contain $P'$, add $P'$ to $\mathscr P_i$. When this loop is finished, filter the corner-connected polyominoes which are not in compliance with Lemma \ref{lmm:submosaics} out of $\mathscr P_n$. The resulting set is exactly the corner-connected $n$-tile polyominoes in compliance with Lemma \ref{lmm:submosaics}.
\end{lemma}

\begin{proof}
    Assume that $\mathscr P_i$ contains all of the corner-connected $i$-tile polyominoes. We claim $\mathscr P_{i+1}$ contains all of the corner-connected $(i+1)$-tile polyominoes. If this were not so, there would be a corner-connected $(i+1)$-tile polyomino with the property that removing any tile makes it no longer corner-connected. Then the graph with a vertex for each tile and an edge for each pair of tiles sharing an edge or corner is a connected graph with more than one vertex with the property that removing any vertex makes it no longer connected. But no such graph exists.

    The studious reader will rightfully object that $\mathscr P_2$ does \emph{not} contain all of the corner-connected dominoes. This is okay because the only extra polyominoes enumerated from including the non-edge-connected domino in $\mathscr P_2$ contain no copies of the edge-connected domino in them, which makes them not in compliance with Lemma \ref{lmm:submosaics}.
\end{proof}

Polyominoes as we have defined them are quotients, and to minimize waiting one must implement the algorithm in a way that works with quotients in an efficient manner. An optimized Python script is available on the author's website.

\subsection*{The Classification}

\begin{theorem}
    If $L$ has no unlinked, unknotted components and $t_C(L) < 12$, then exactly one holds:
    \begin{itemize}
        \item $t_C(L) = 6$ and $L$ is the Hopf link $P(1,1)$,
        \item $t_C(L) = 8$ and $L$ is the trefoil knot $3_1$ or Solomon's knot $P(1,1,1,1)$,
        \item $t_C(L) = 10$ and $L$ is the connect sum of two Hopf links $P(1,1) \# P(1,1)$, the cinquefoil knot $5_1$, or the star of David link $P(1,1,1,1,1,1)$,
        \item $t_C(L) = 11$ and $L$ is the figure-eight knot $4_1$ or the three-twist knot $5_2$.
    \end{itemize}
    Table \ref{tbl:examples} shows explicit examples of mosaics representing the tile number for each of these links.
    \label{thm:classification}
\end{theorem}

\begin{proof}
    It is easy to see that $t_C(L)$ cannot be 1, 2, or 3. We can therefore be sure that for $t_C(L) \leq 7$, all mosaics in $\mathscr M_C^L$ are corner-connected, for otherwise we could split the mosaic into two corner-connected components, each representing links with no unlinked, unknotted components. Running the algorithm of Lemma \ref{lmm:algorithm} revealed that there are no corner-connected polynominoes with 4 or 5 tiles in compliance with Lemma \ref{lmm:submosaics}. Therefore, this argument can be extended to say that if $t_C(L) < 12$, all mosaics in $\mathscr M_C^L$ are corner-connected. (The argument can be pushed no further, as the conclusion is false for $t_C(L) = 12$: consider $L$ two unlinked Hopf links.)
    
    The algorithm of Lemma \ref{lmm:algorithm} found that the corner-connected polynominoes with 6, 7, 8, 9, 10, or 11 tiles in compliance with Lemma \ref{lmm:submosaics} are exactly the 35 depicted in Table \ref{tbl:polyominolist}. We want to know which of them are polyominoifications of mosaics in $\mathscr M_C^L$.
    
    \begin{table}
        \centering
        \begin{tabular}{cM{18.5mm}M{18.5mm}M{18.5mm}M{18.5mm}M{18.5mm}}
        & a & b & c & d & e\\
            1
            & \includegraphics[width=.15\textwidth]{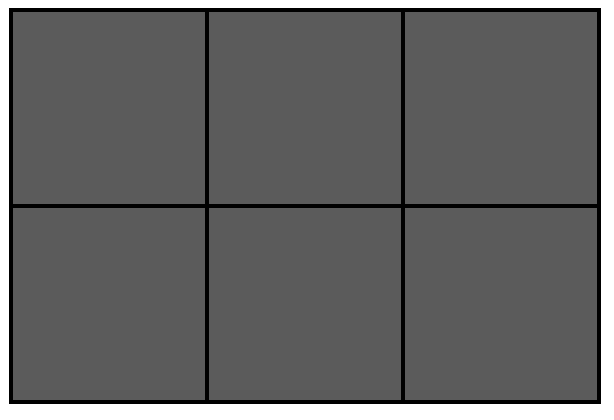}
            & \includegraphics[width=.15\textwidth]{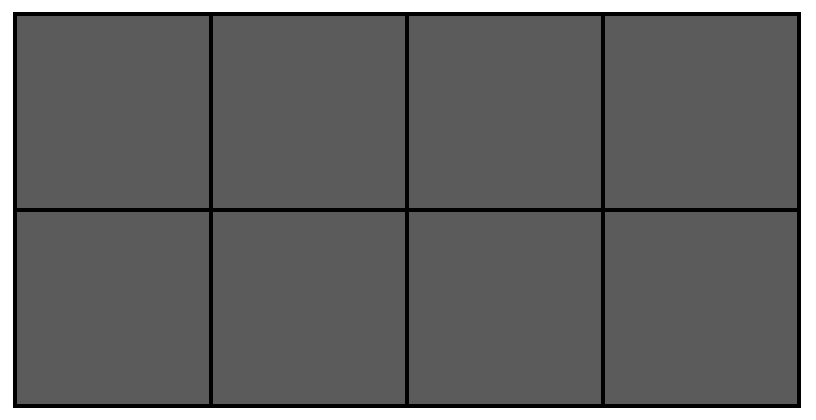}
            & \includegraphics[width=.15\textwidth]{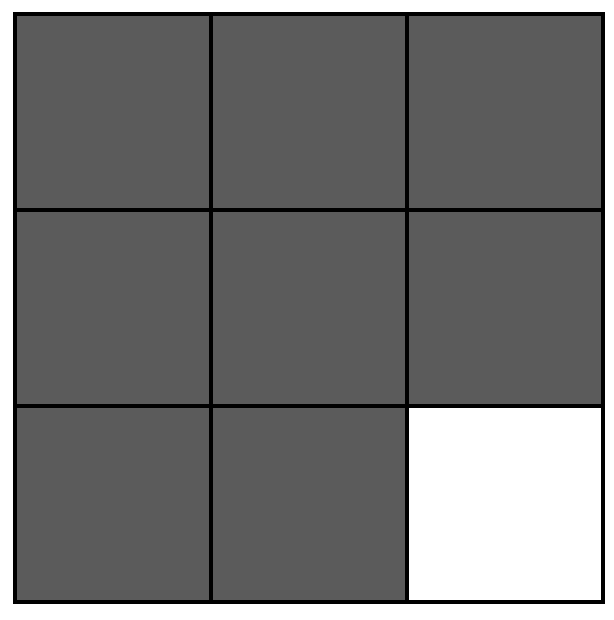}
            & \includegraphics[width=.15\textwidth]{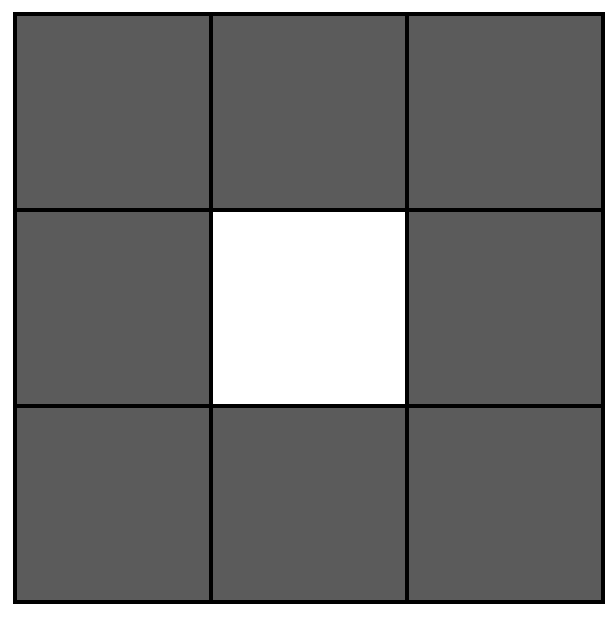}
            & \includegraphics[width=.15\textwidth]{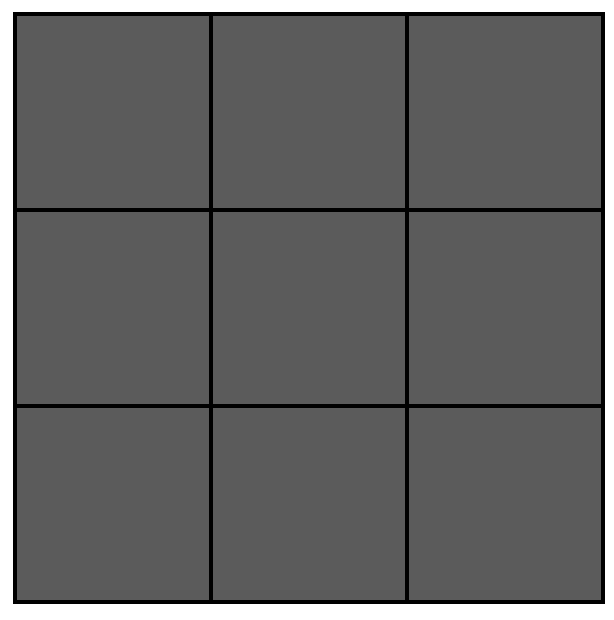} \\
            2
            & \includegraphics[width=.15\textwidth]{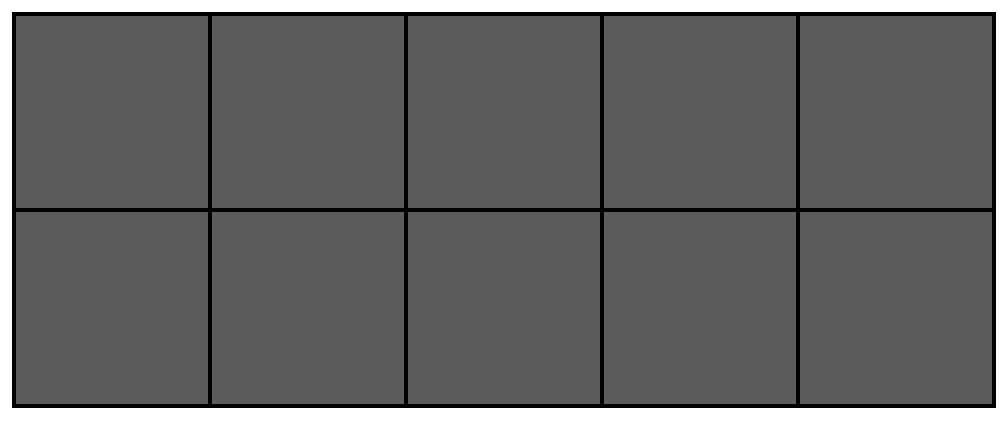}
            & \includegraphics[width=.15\textwidth]{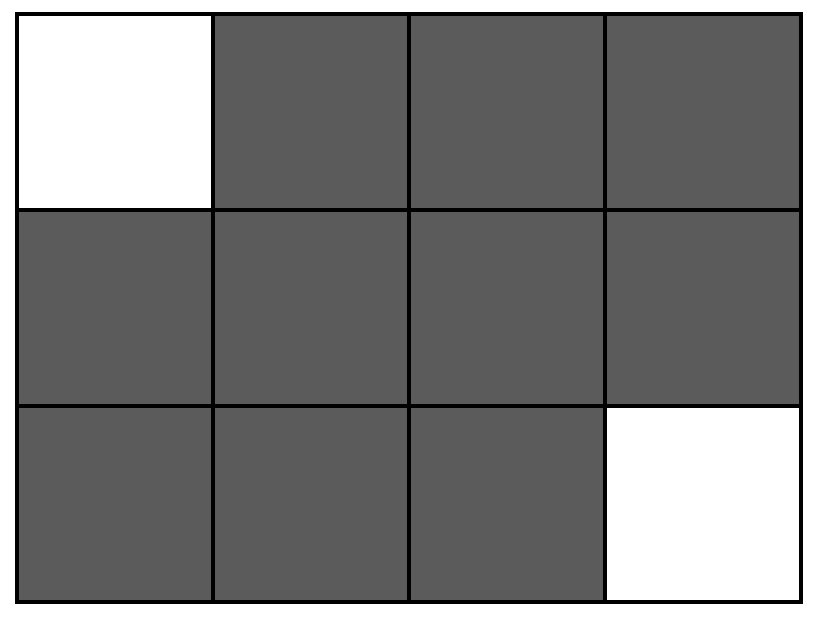}
            & \includegraphics[width=.15\textwidth]{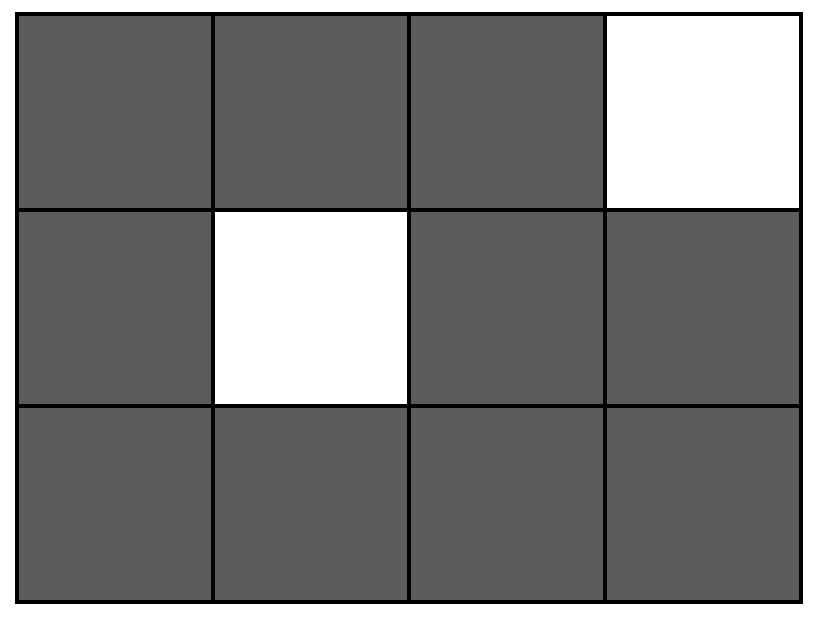}
            & \includegraphics[width=.15\textwidth]{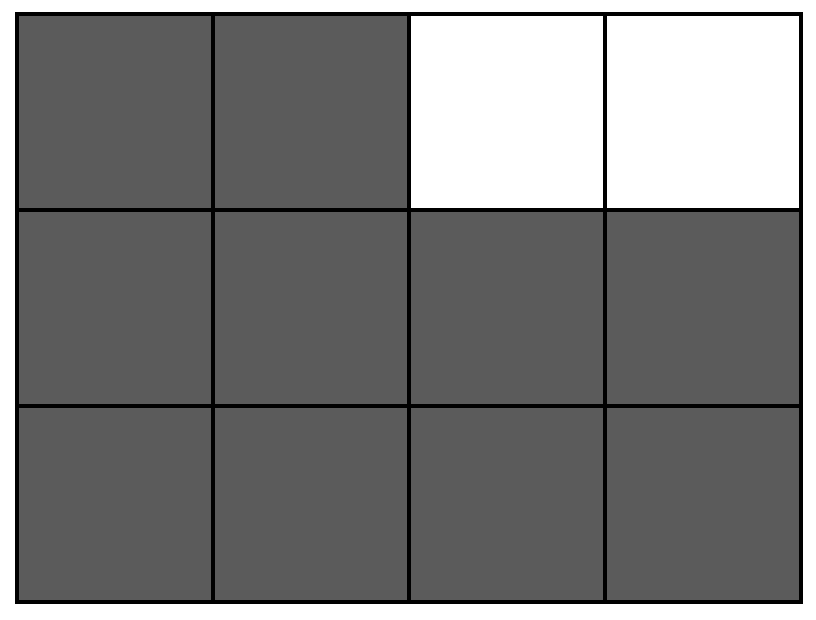}
            & \includegraphics[width=.15\textwidth]{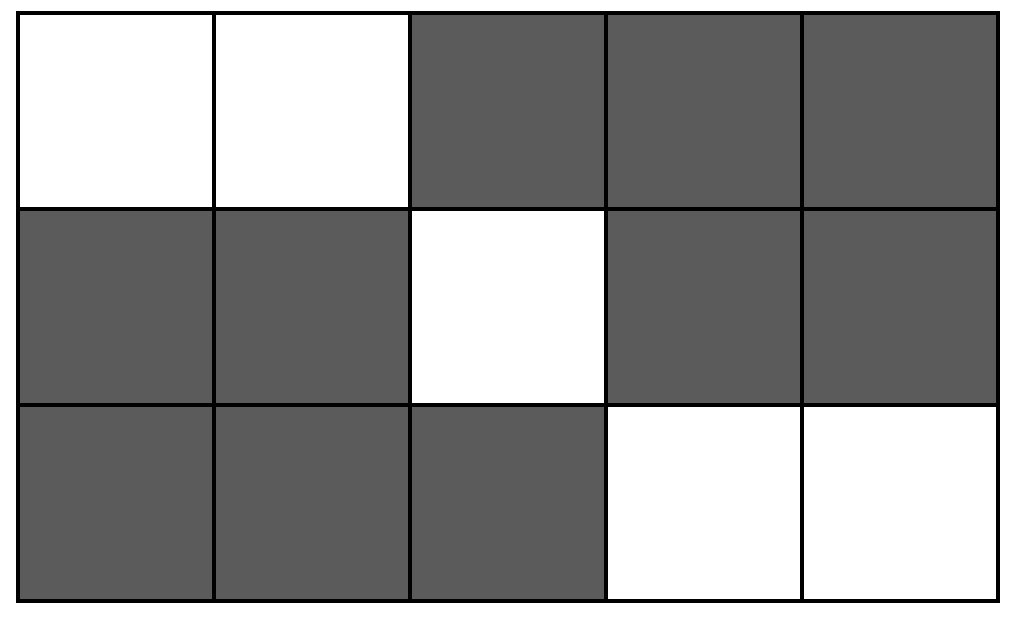} \\
            3
            & \includegraphics[width=.15\textwidth]{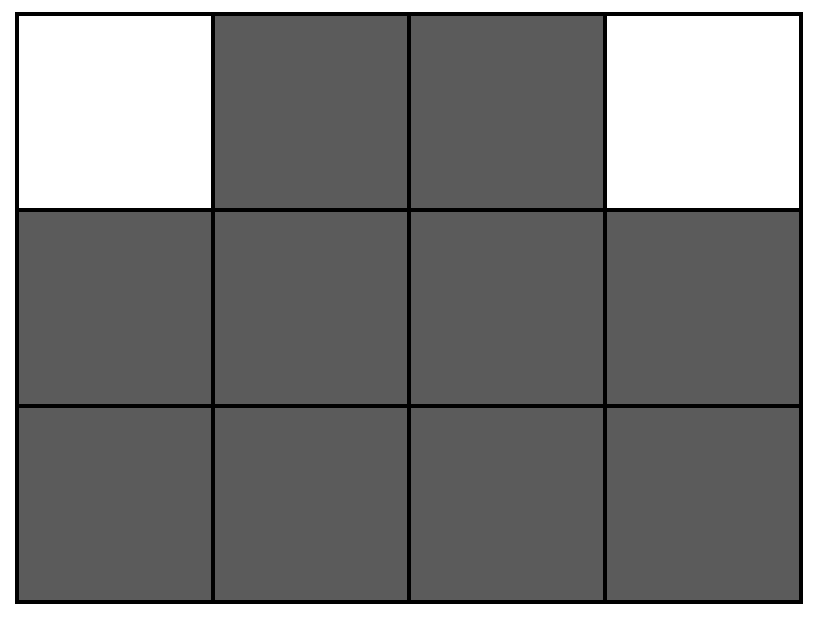}
            & \includegraphics[width=.15\textwidth]{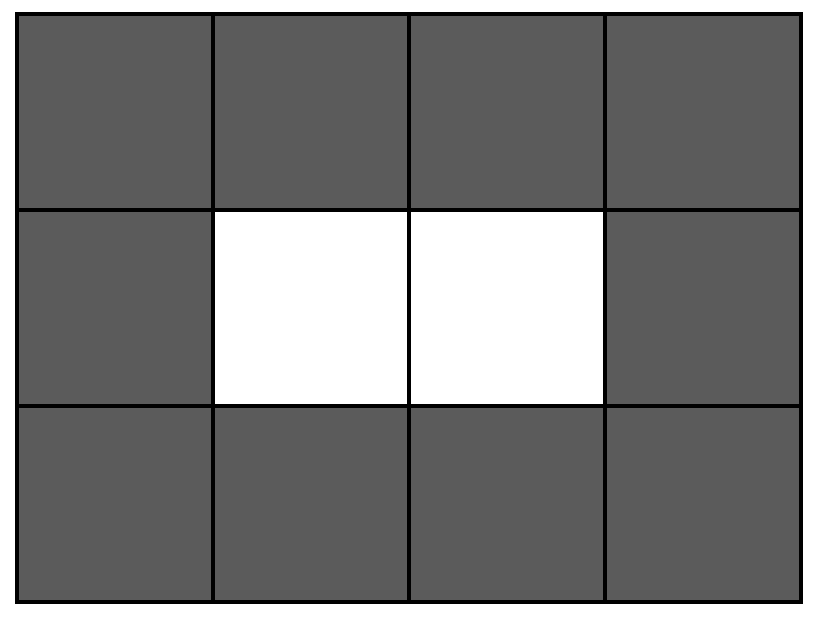}
            & \includegraphics[width=.15\textwidth]{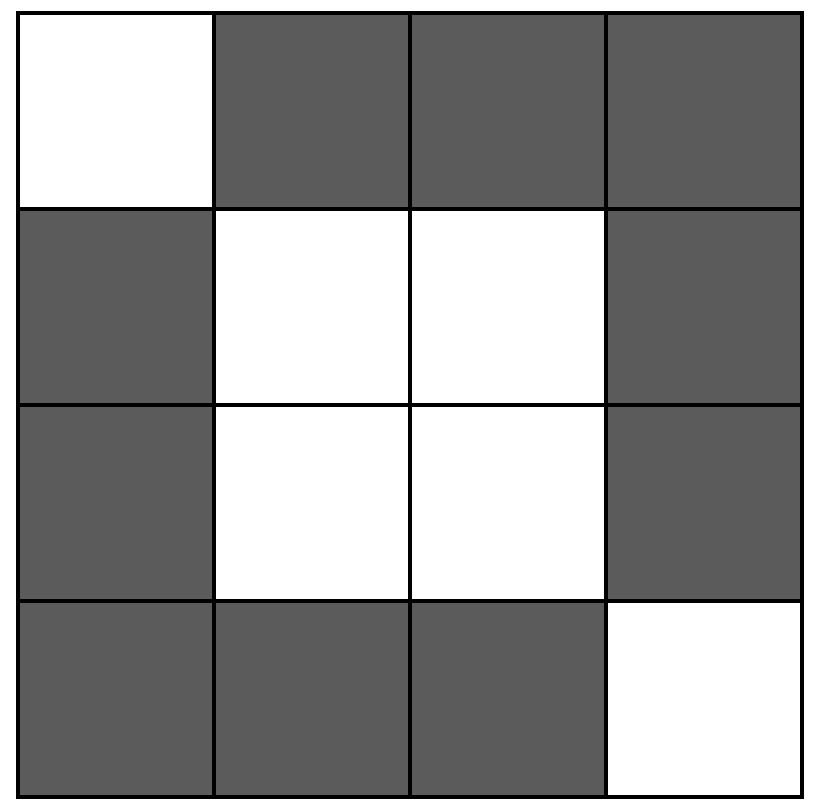}
            & \includegraphics[width=.15\textwidth]{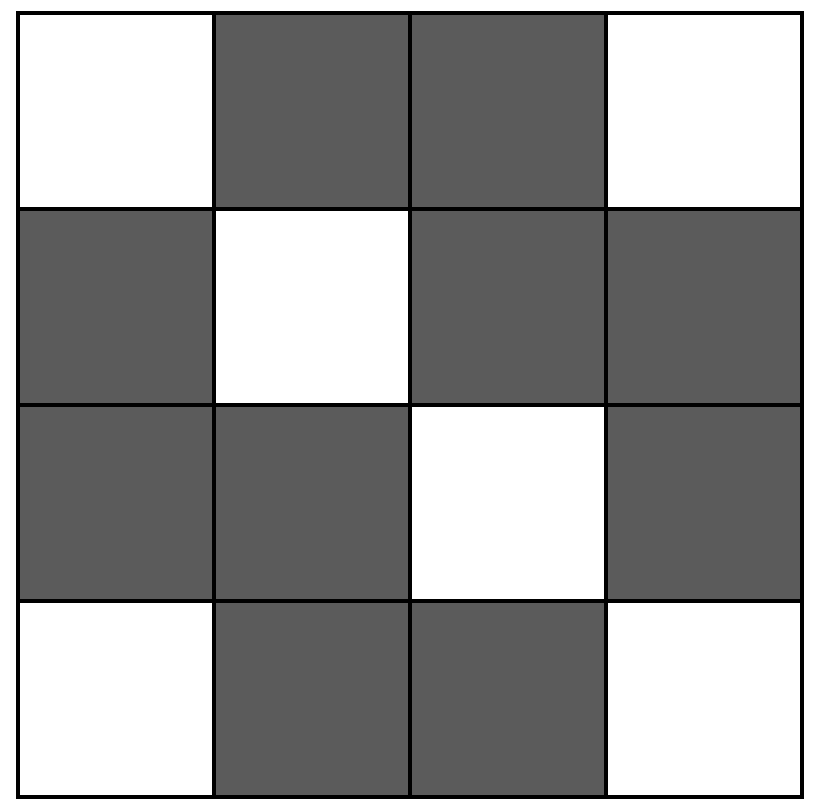}
            & \includegraphics[width=.15\textwidth]{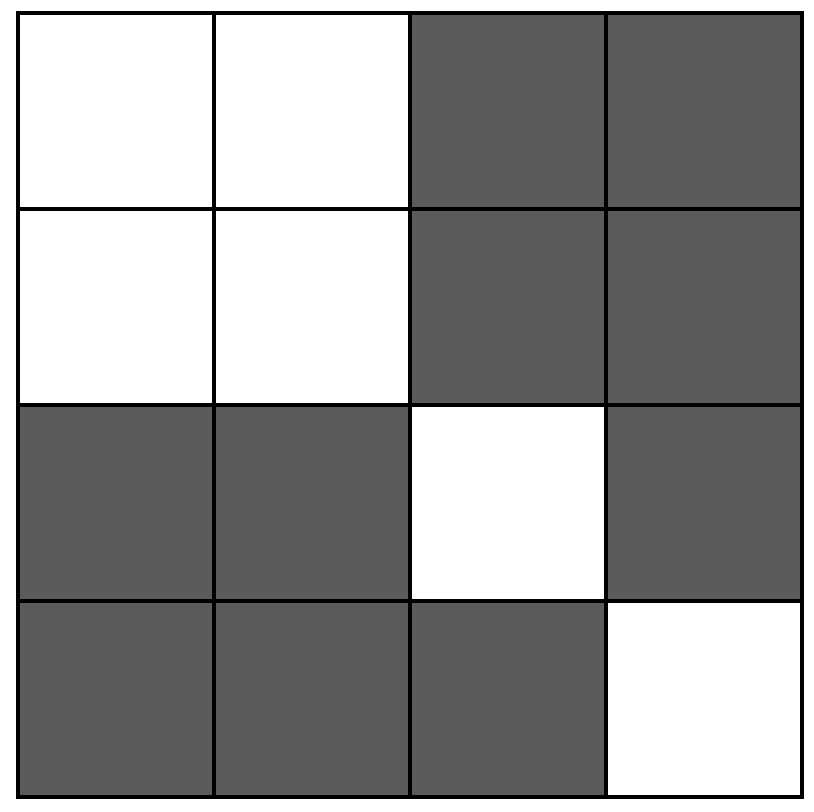} \\
            4
            & \includegraphics[width=.15\textwidth]{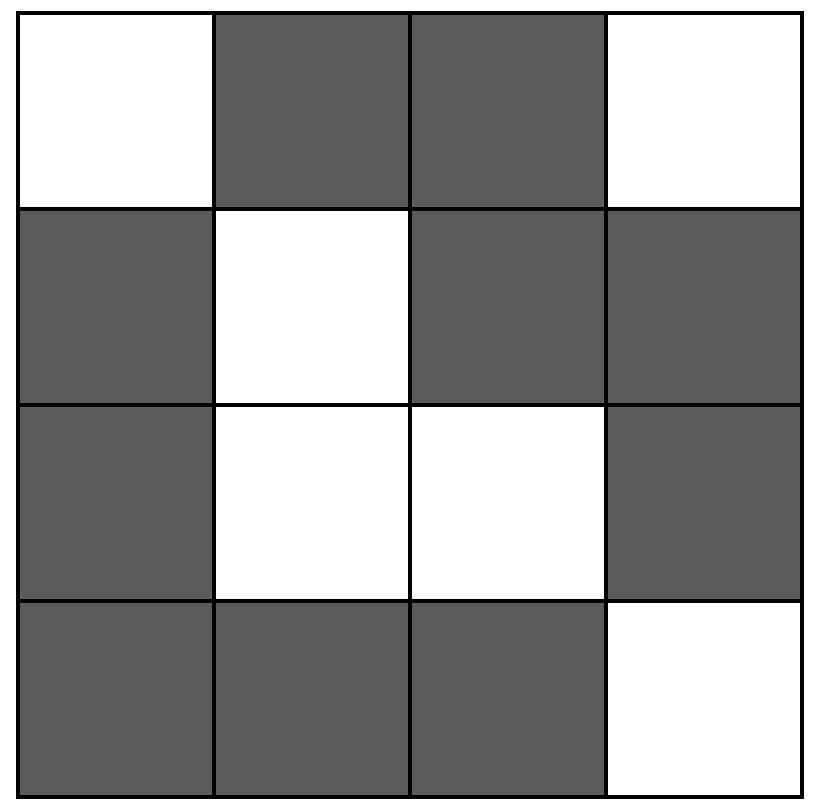}
            & \includegraphics[width=.15\textwidth]{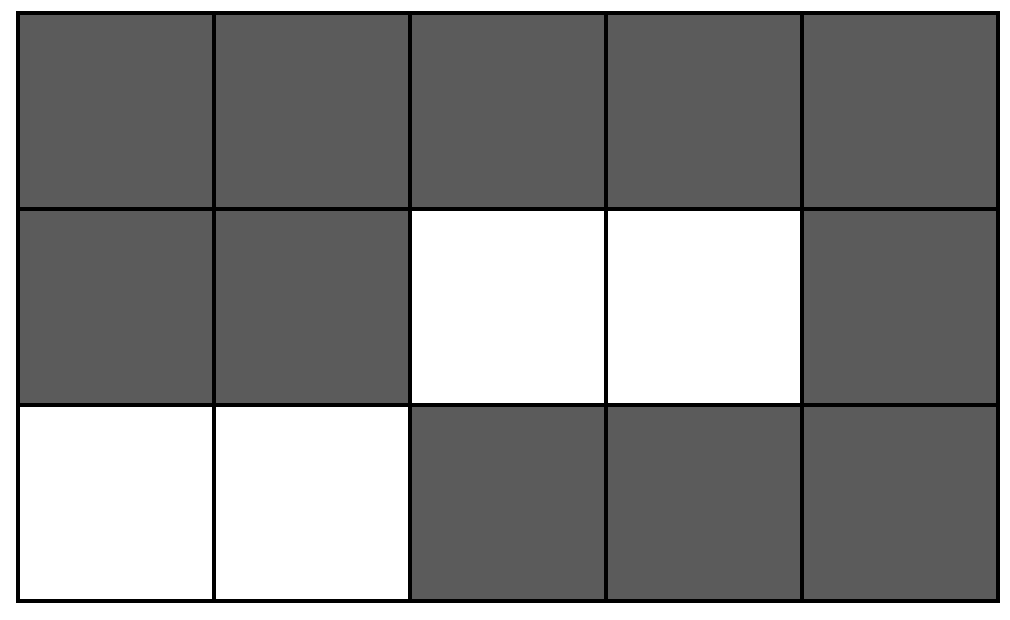}
            & \includegraphics[width=.15\textwidth]{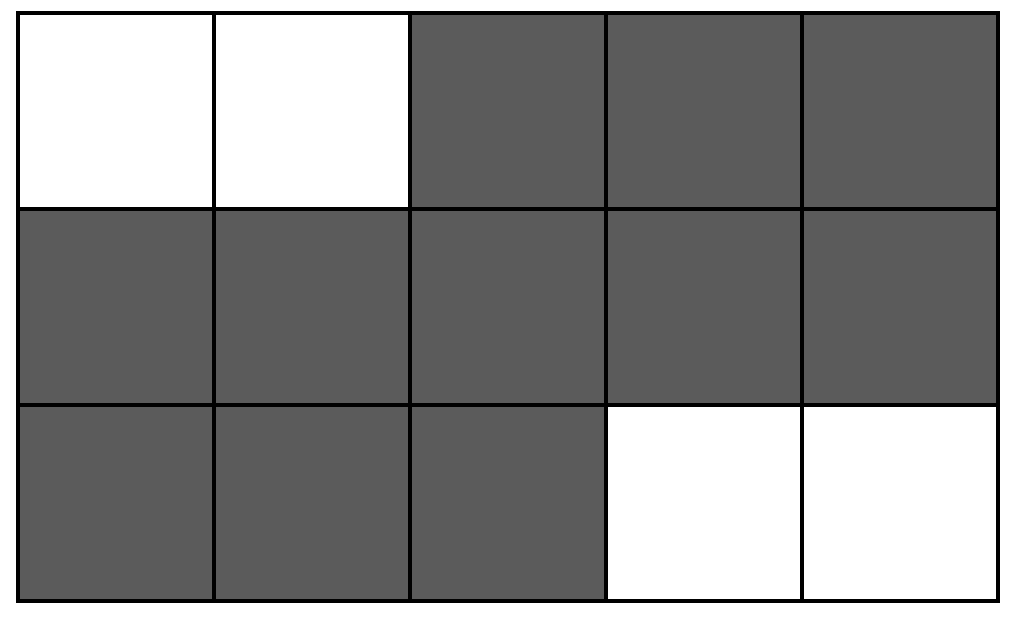}
            & \includegraphics[width=.15\textwidth]{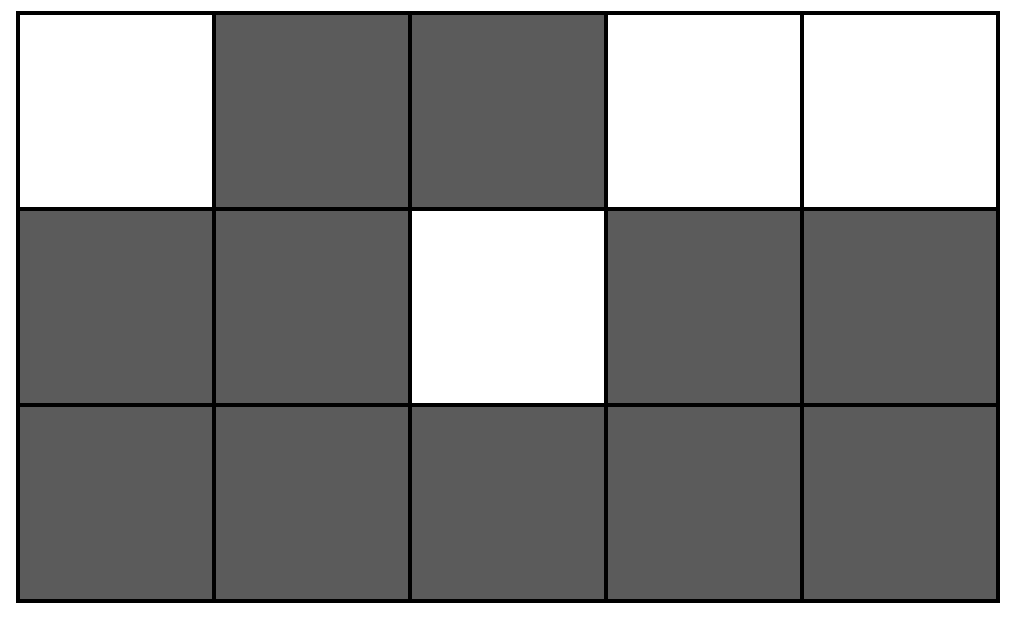}
            & \includegraphics[width=.15\textwidth]{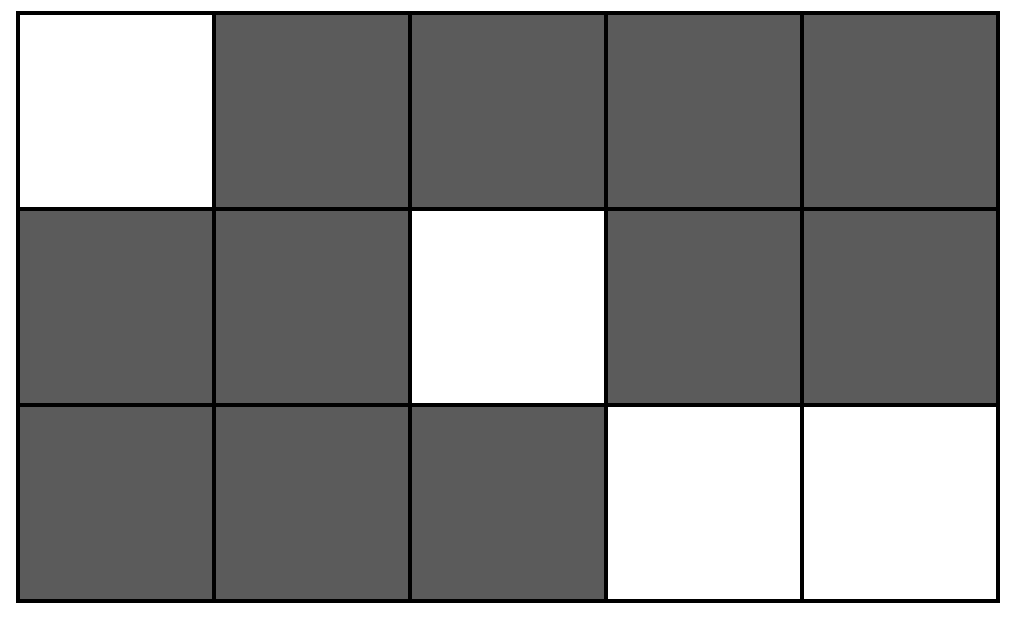} \\
            5
            & \includegraphics[width=.15\textwidth]{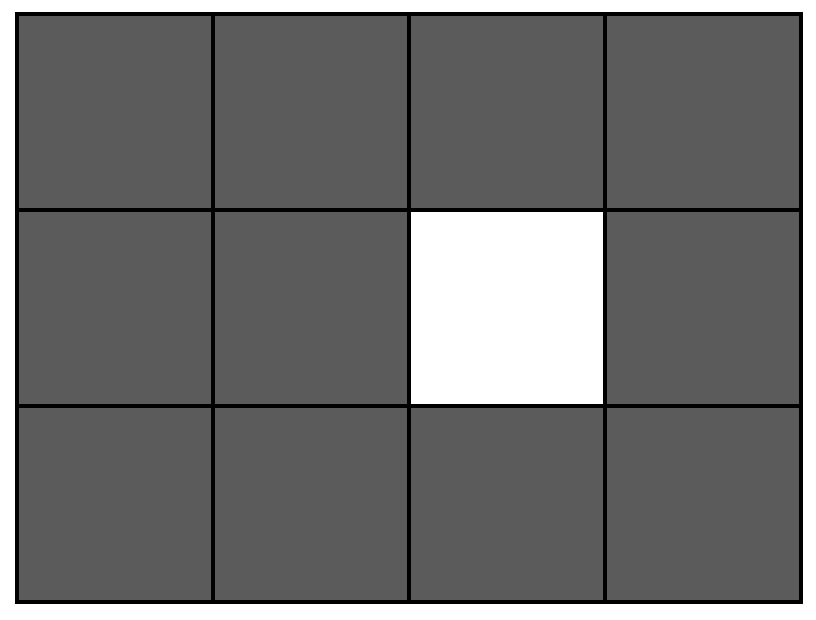}
            & \includegraphics[width=.15\textwidth]{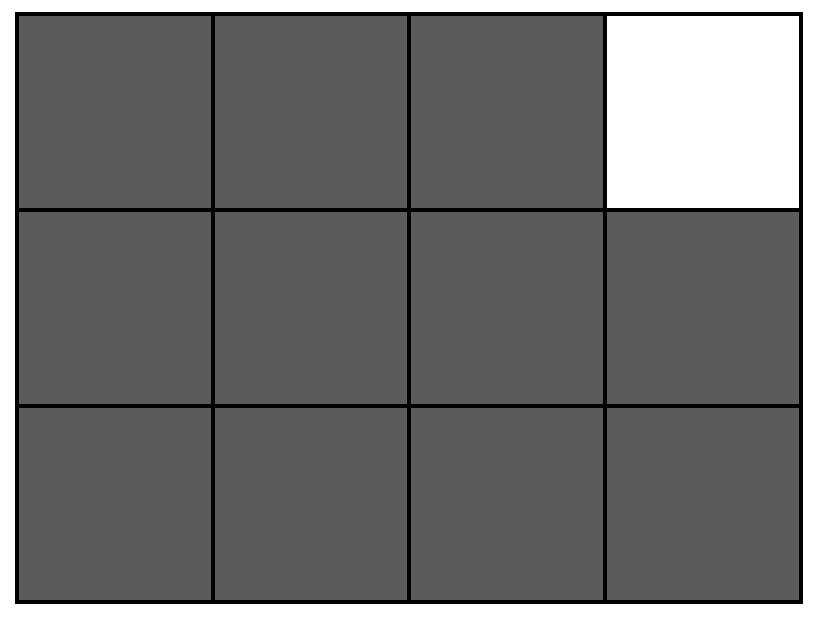}
            & \includegraphics[width=.15\textwidth]{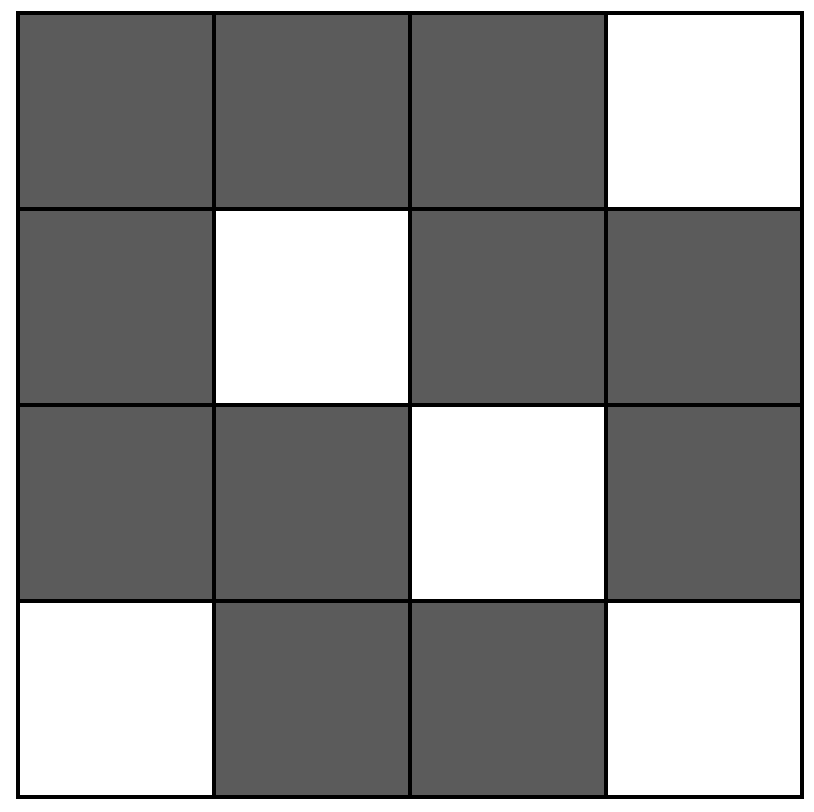}
            & \includegraphics[width=.15\textwidth]{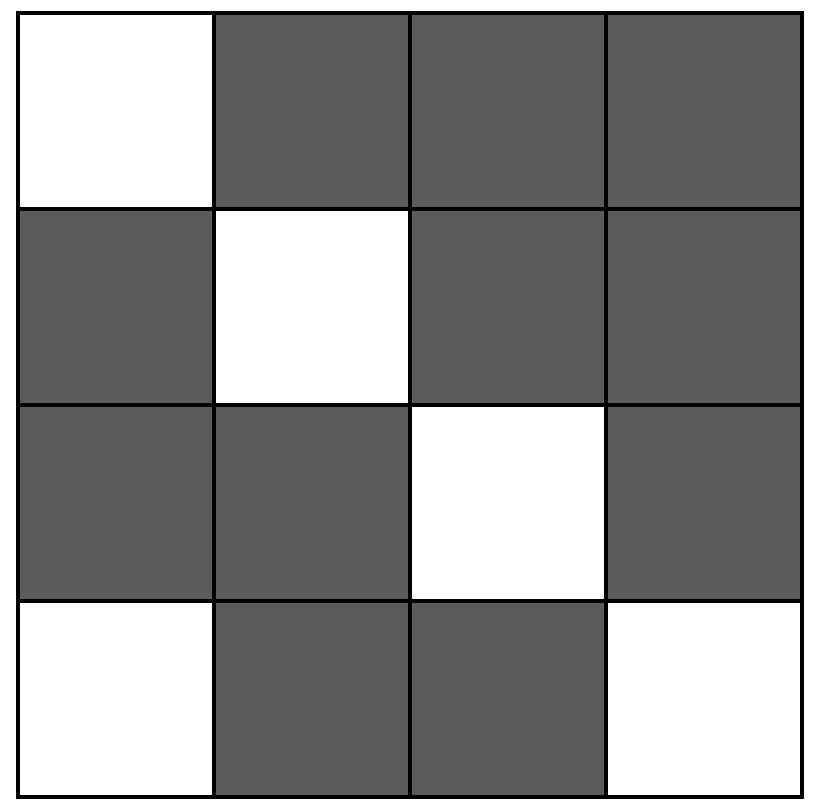}
            & \includegraphics[width=.15\textwidth]{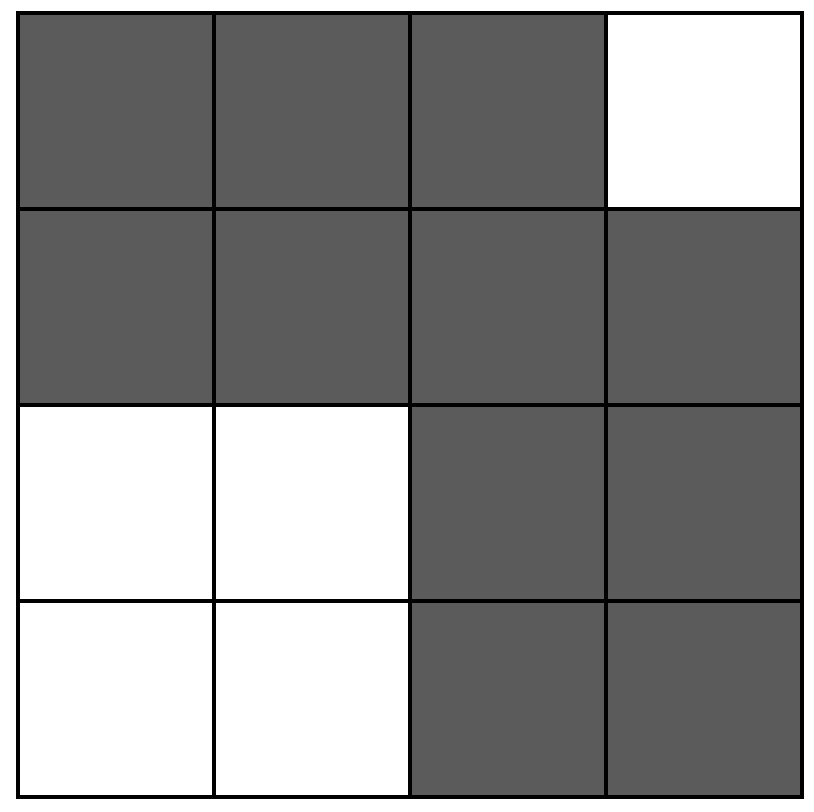} \\
            6
            & \includegraphics[width=.15\textwidth]{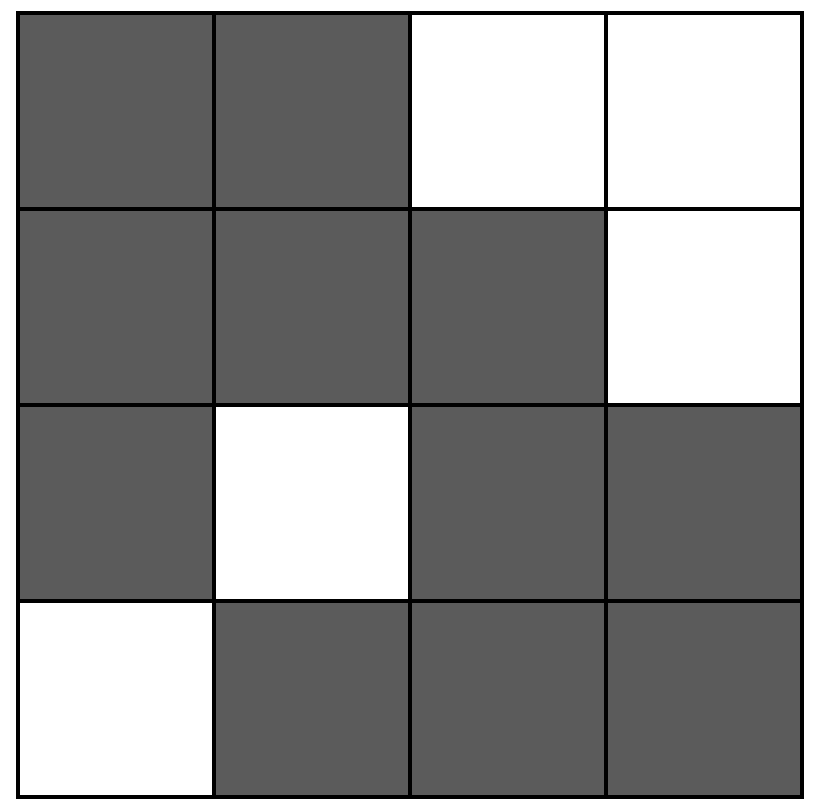}
            & \includegraphics[width=.15\textwidth]{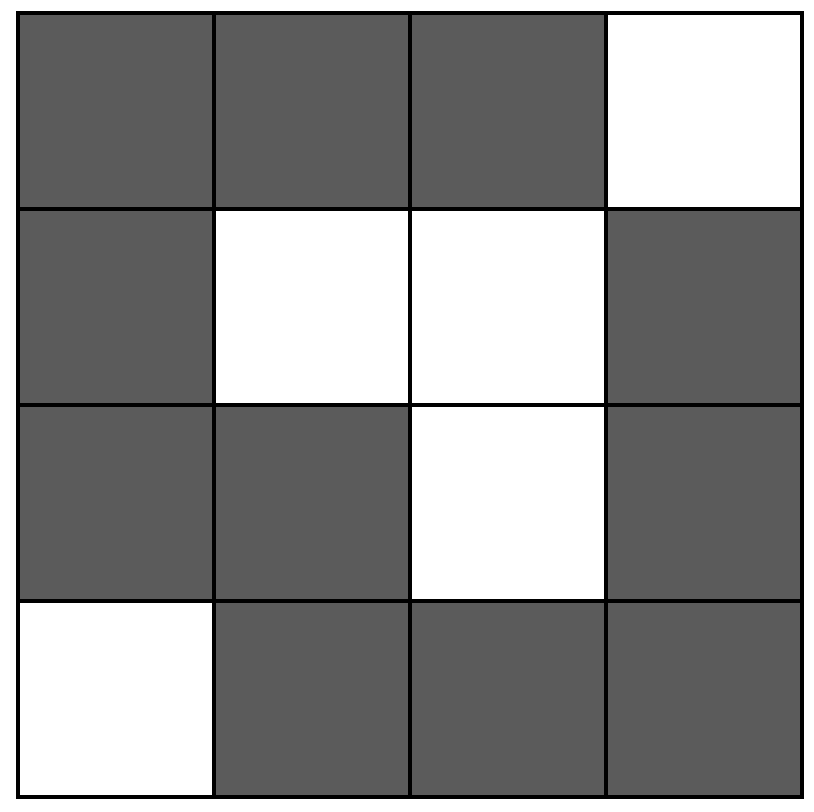}
            & \includegraphics[width=.15\textwidth]{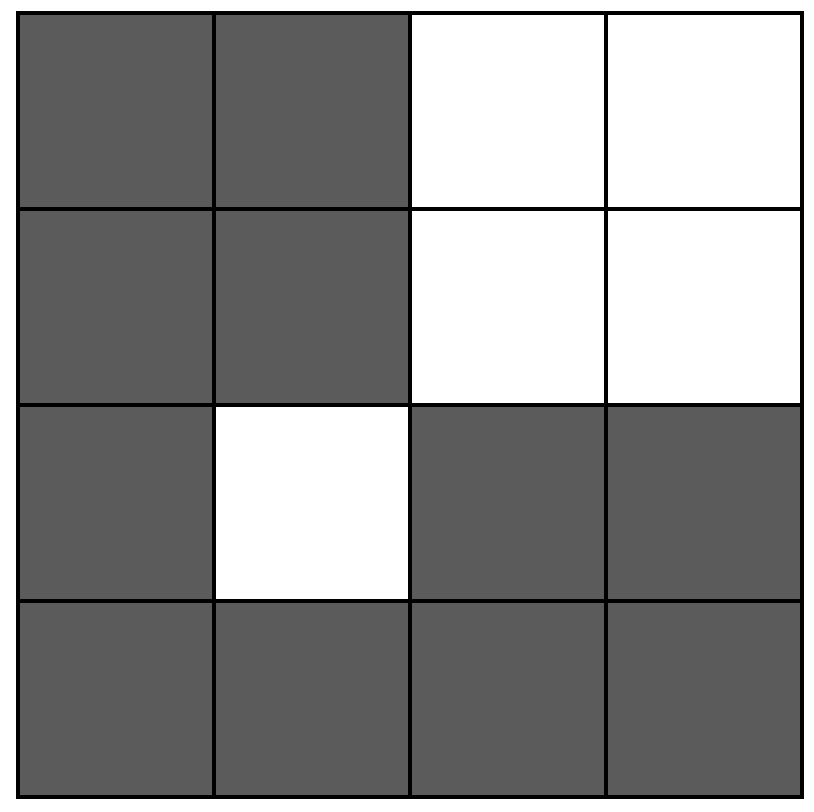}
            & \includegraphics[width=.15\textwidth]{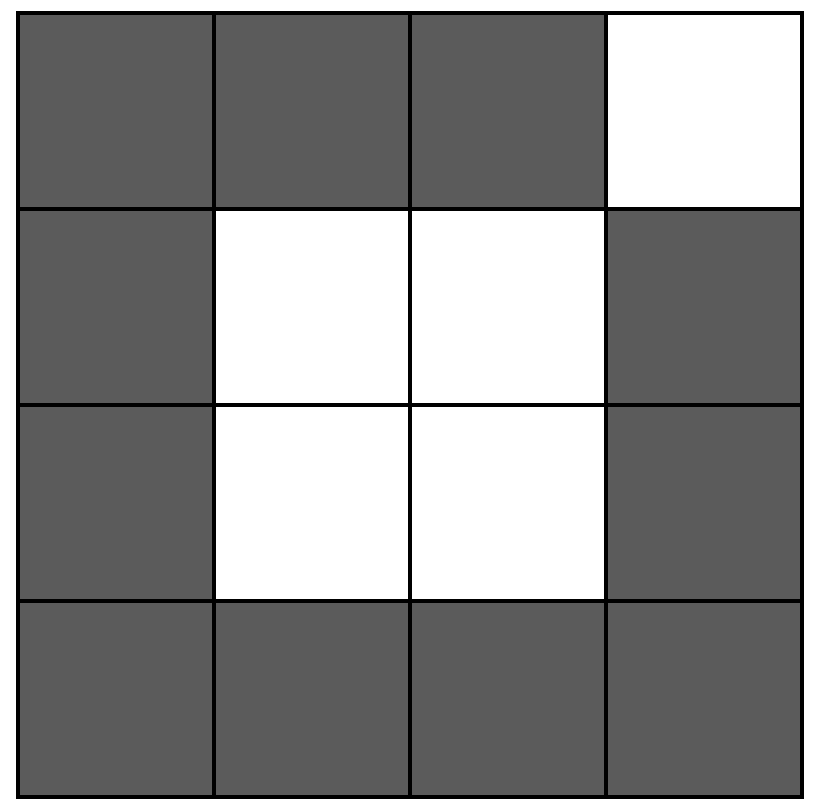}
            & \includegraphics[width=.15\textwidth]{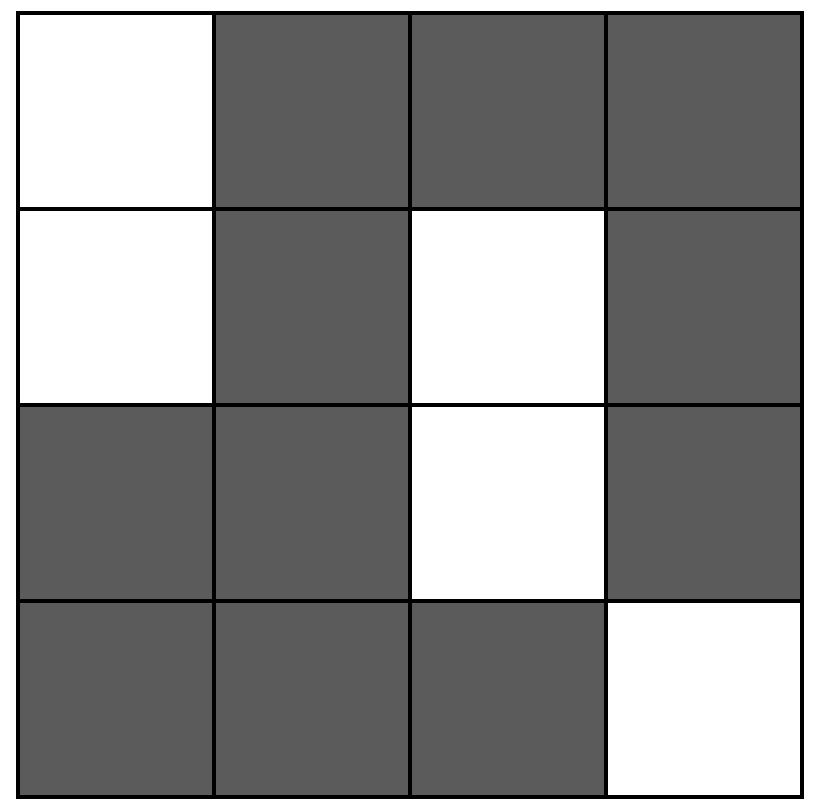} \\
            7
            & \includegraphics[width=.15\textwidth]{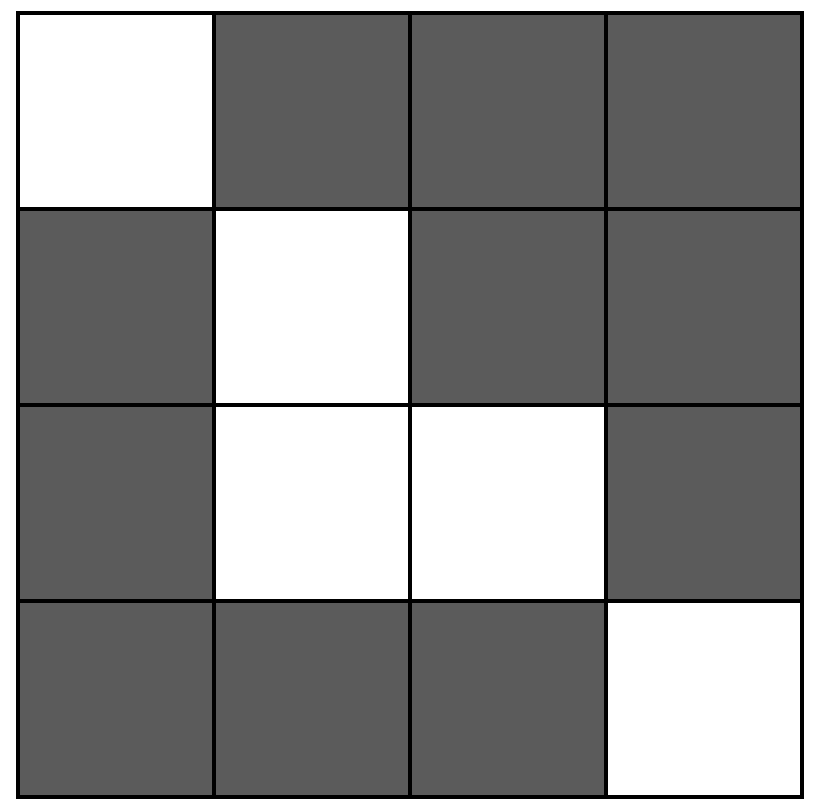}
            & \includegraphics[width=.15\textwidth]{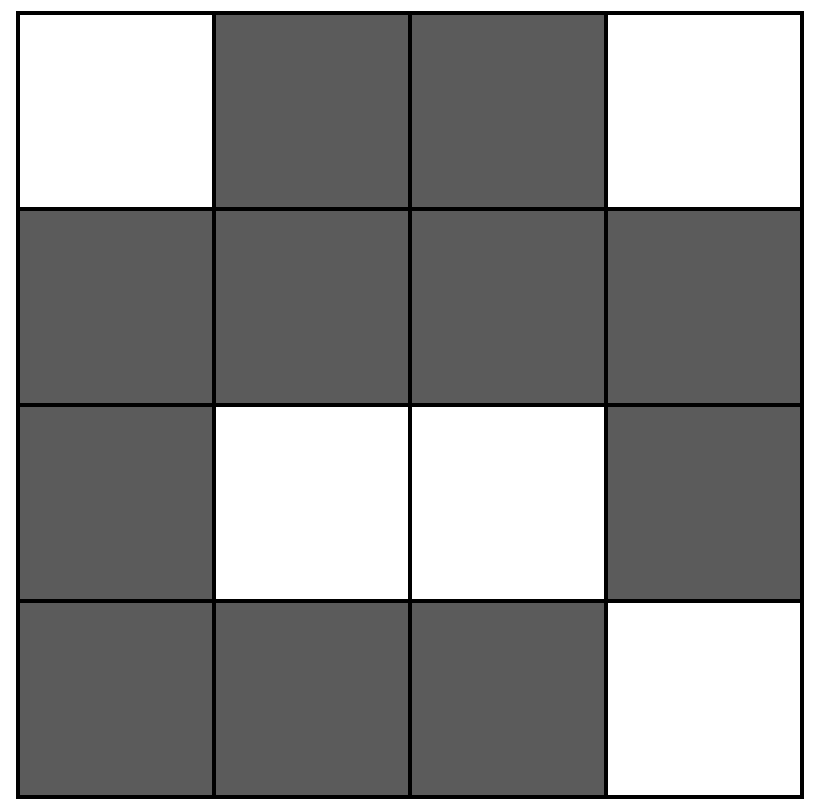}
            & \includegraphics[width=.15\textwidth]{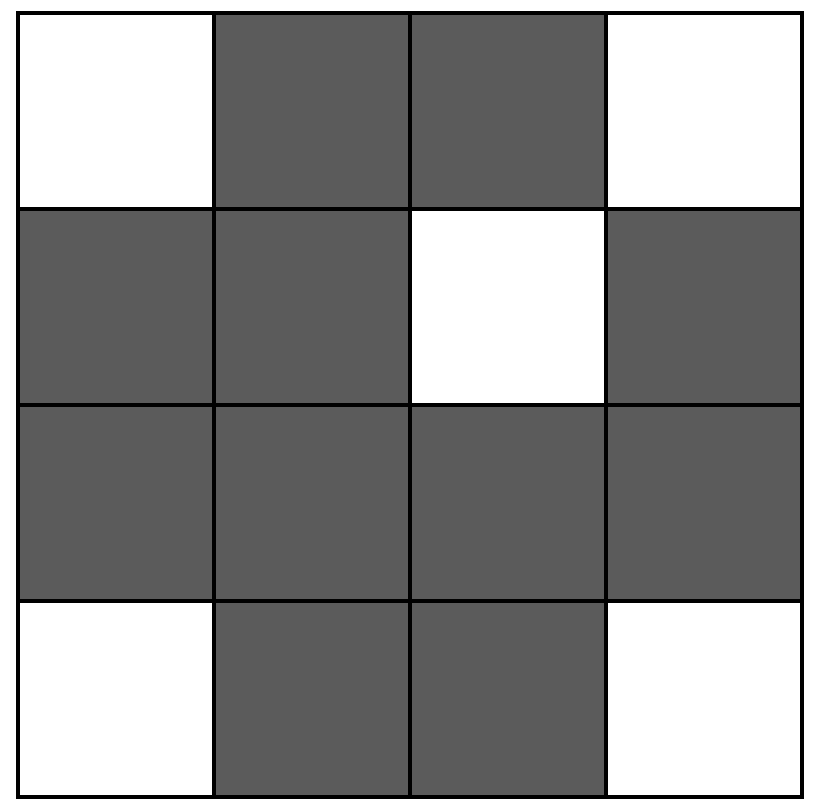}
            & \includegraphics[width=.15\textwidth]{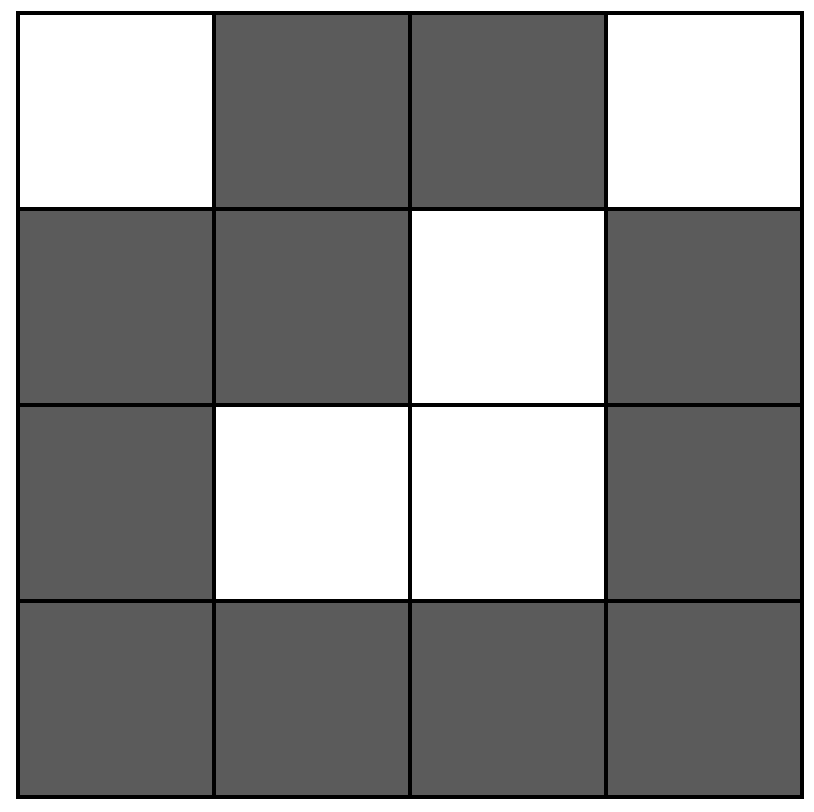}
            & \includegraphics[width=.15\textwidth]{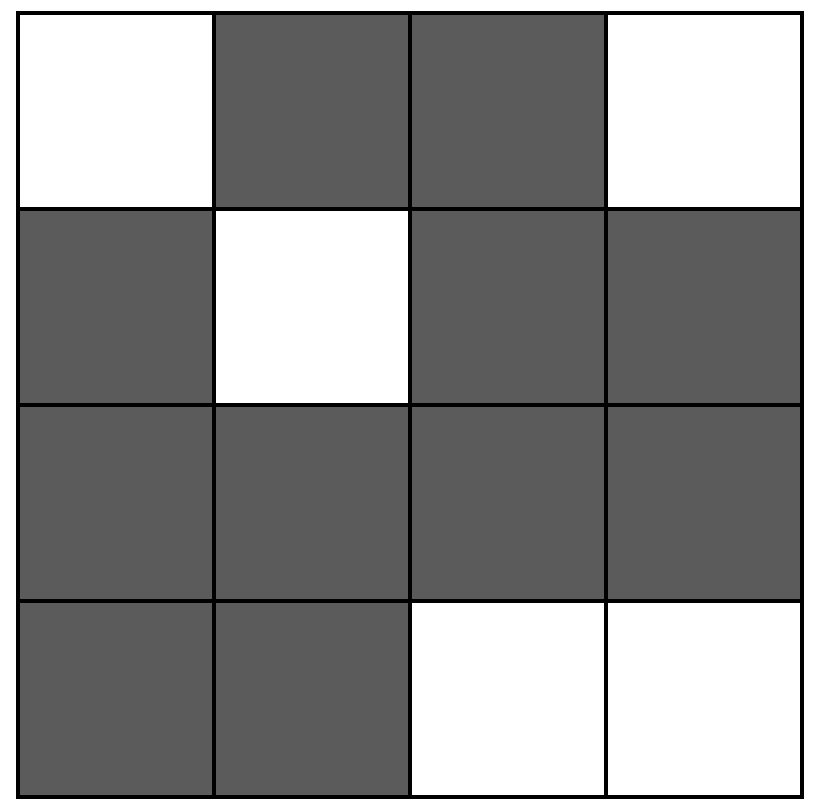} \\
        \end{tabular}
        \caption{An exhaustive list of all corner-connected polyominoes with 6, 7, 8, 9, 10, or 11 tiles in compliance with Lemma \ref{lmm:submosaics}.}
        \label{tbl:polyominolist}
    \end{table}

    The rest of the proof will combinatorially attack each polyomino in Table \ref{tbl:polyominolist} by using casework to fill in tiles until either a contradiction is reached or a mosaic in $\mathscr M_C^L$ has been constructed whose polyominoificiation is the polyomino. We explain how to carry out the computation but omit many details. Applying Figure \ref{fig:middleist5} to the entries of Table \ref{tbl:polyominolist} repeatedly yields Table \ref{tbl:polyominolistcorners}.
    
    \begin{table}
        \centering
        \begin{tabular}{cM{18.5mm}M{18.5mm}M{18.5mm}M{18.5mm}M{18.5mm}}
        & a & b & c & d & e\\
            1
            & \includegraphics[width=.15\textwidth]{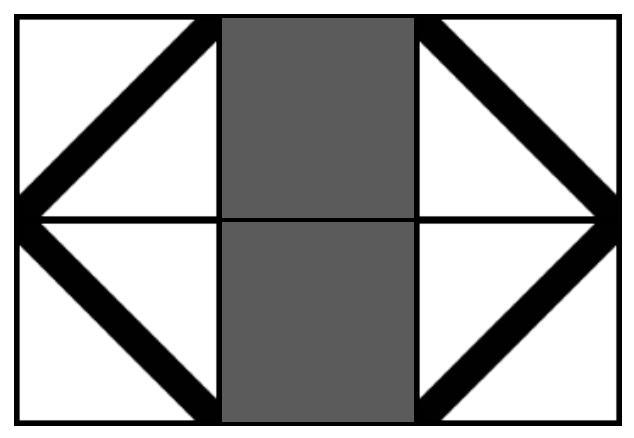}
            & \includegraphics[width=.15\textwidth]{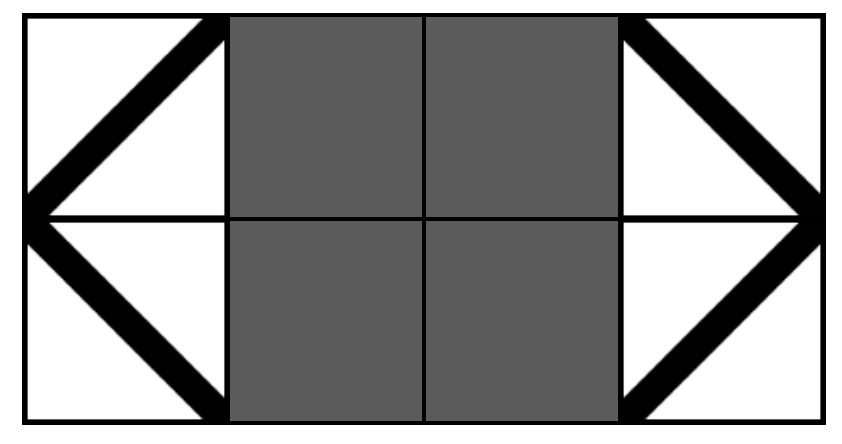}
            & \includegraphics[width=.15\textwidth]{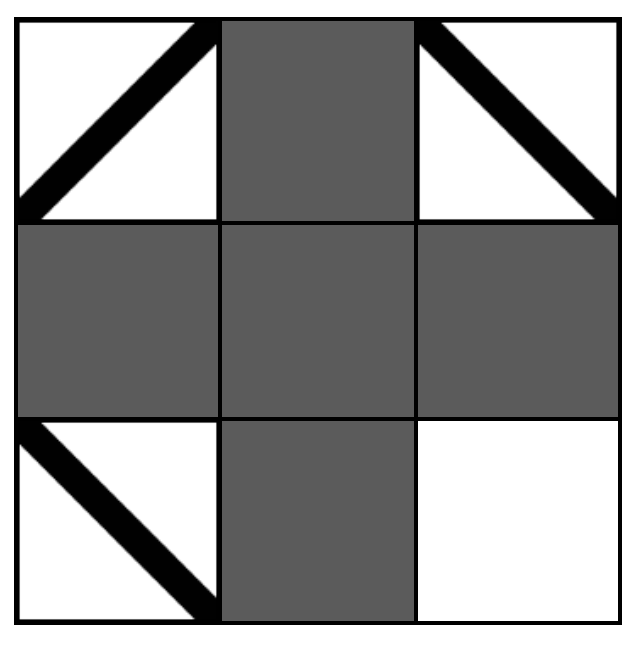}
            & \includegraphics[width=.15\textwidth]{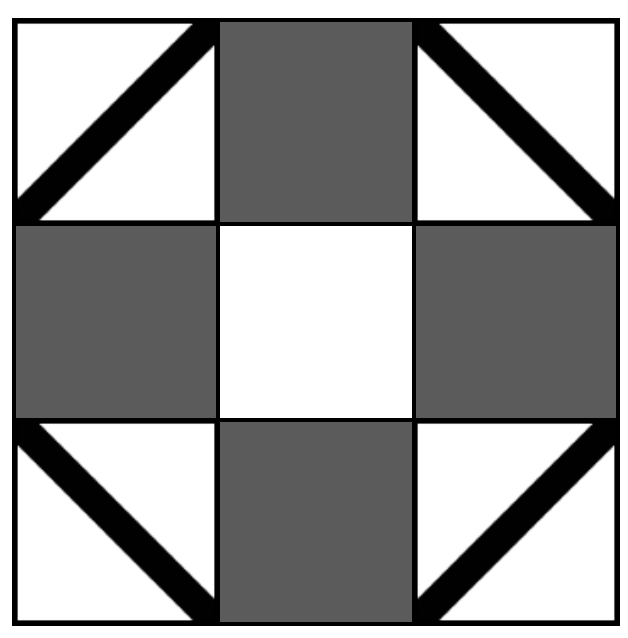}
            & \includegraphics[width=.15\textwidth]{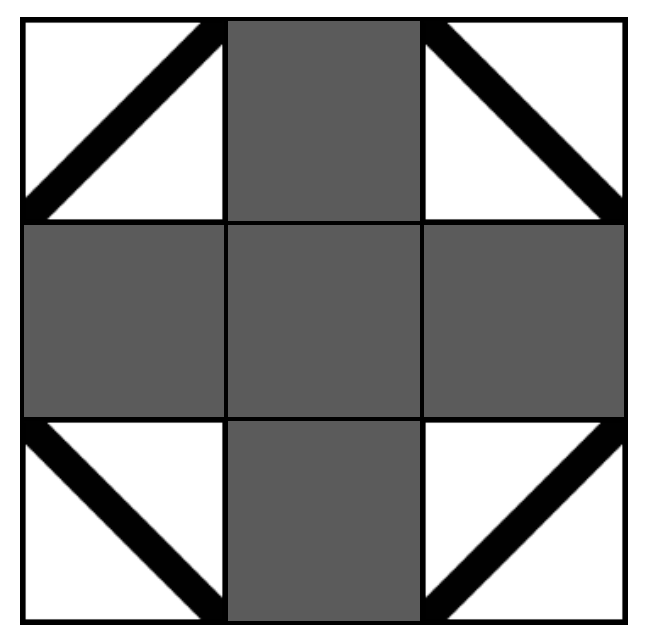} \\
            2
            & \includegraphics[width=.15\textwidth]{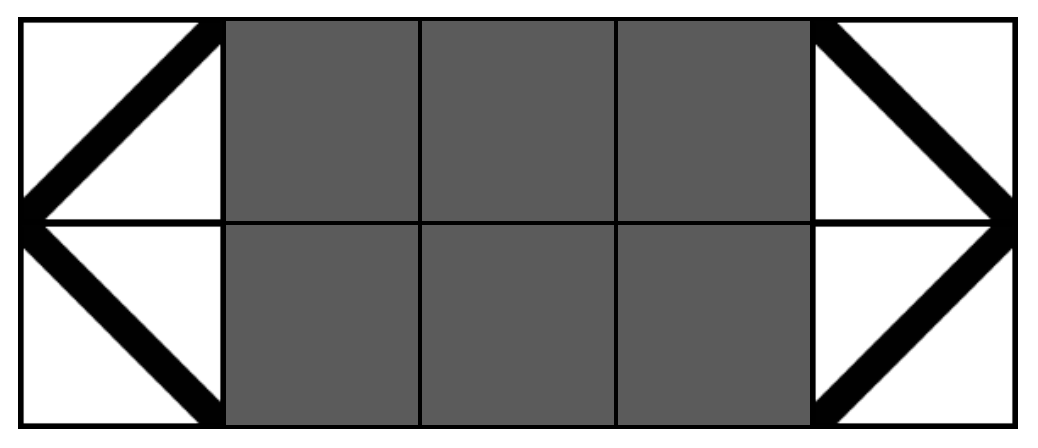}
            & \includegraphics[width=.15\textwidth]{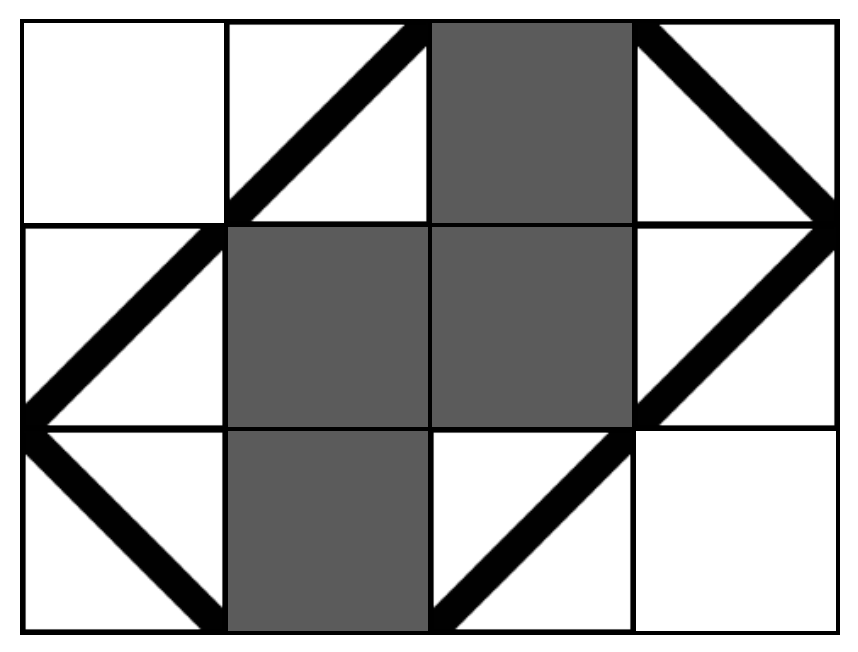}
            & \includegraphics[width=.15\textwidth]{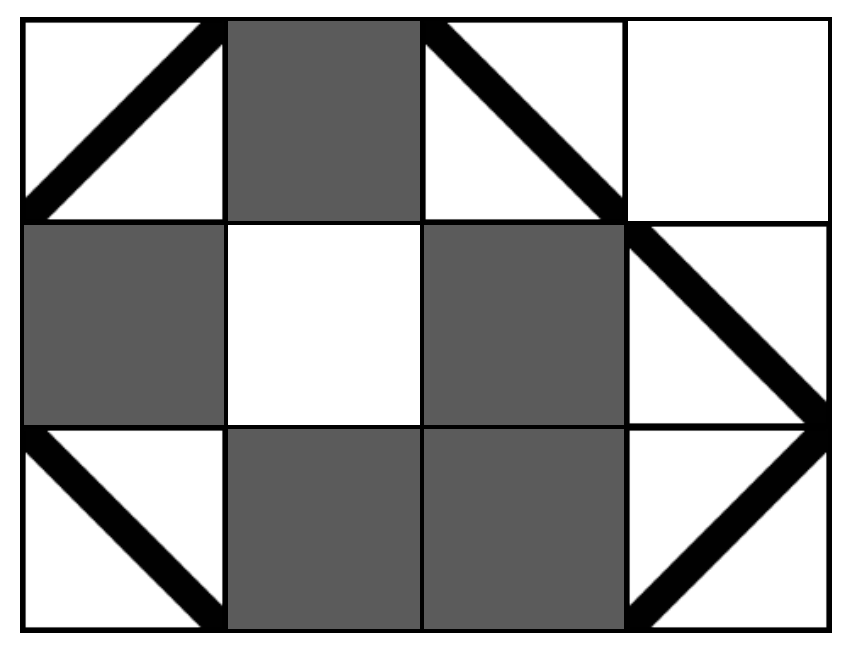}
            & \includegraphics[width=.15\textwidth]{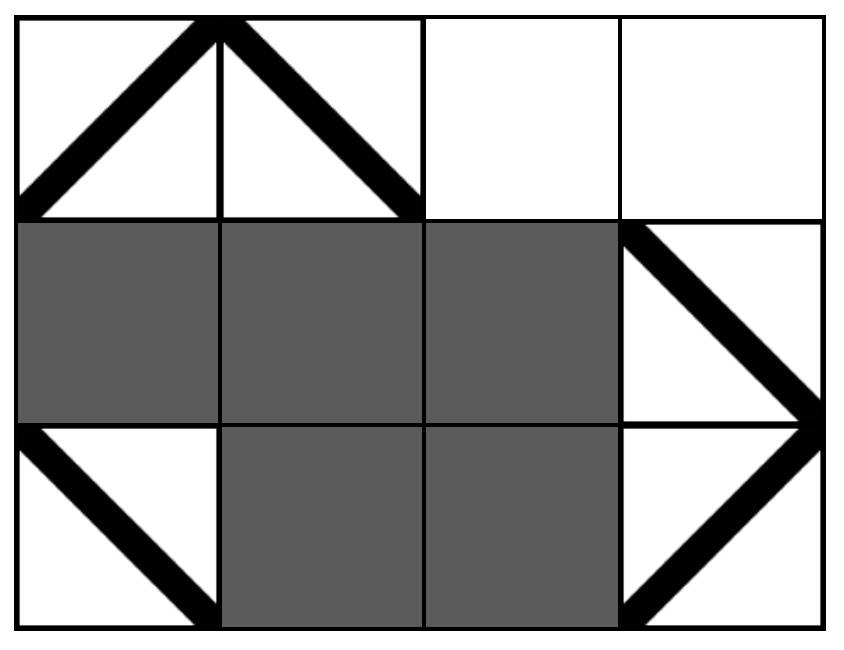}
            & \includegraphics[width=.15\textwidth]{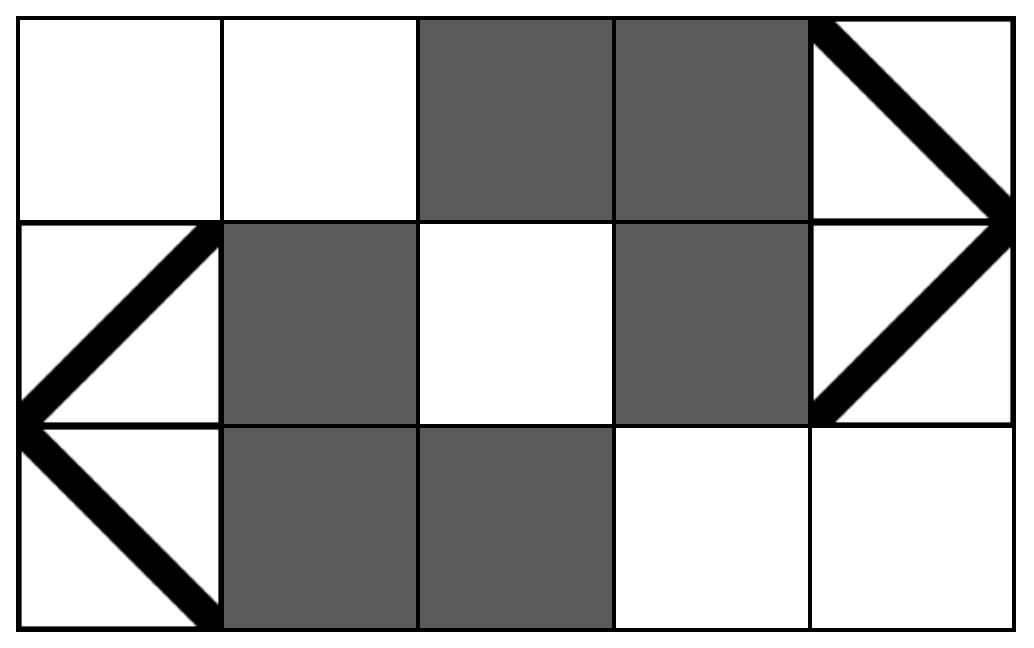} \\
            3
            & \includegraphics[width=.15\textwidth]{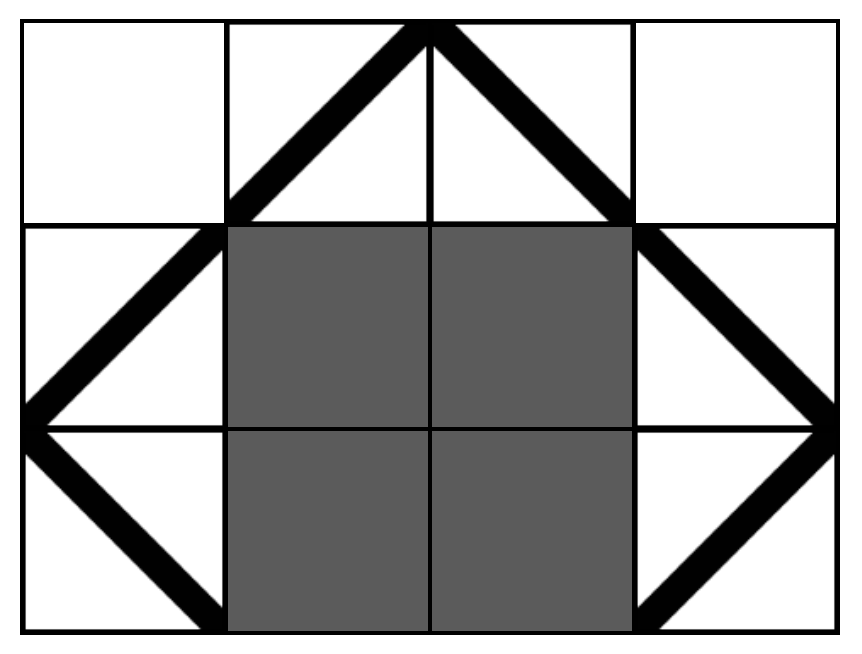}
            & \includegraphics[width=.15\textwidth]{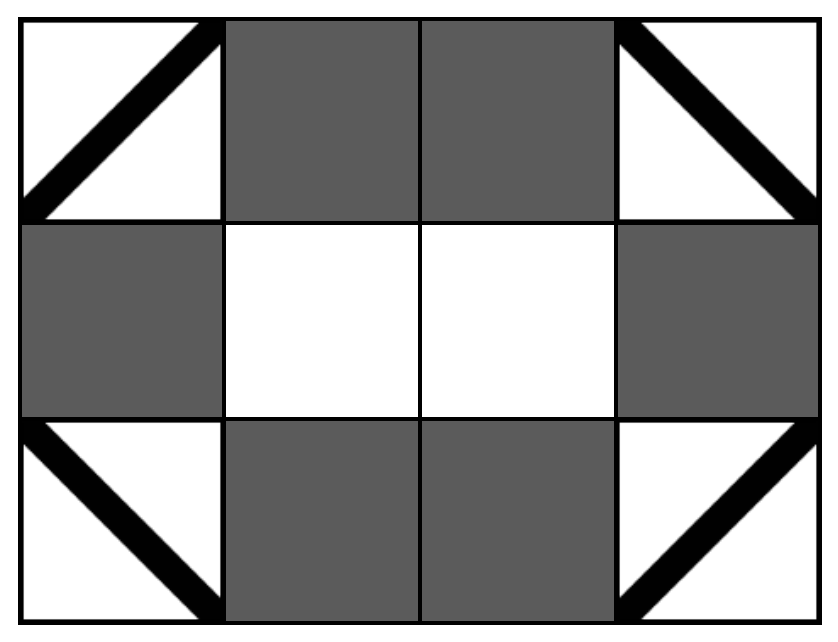}
            & \includegraphics[width=.15\textwidth]{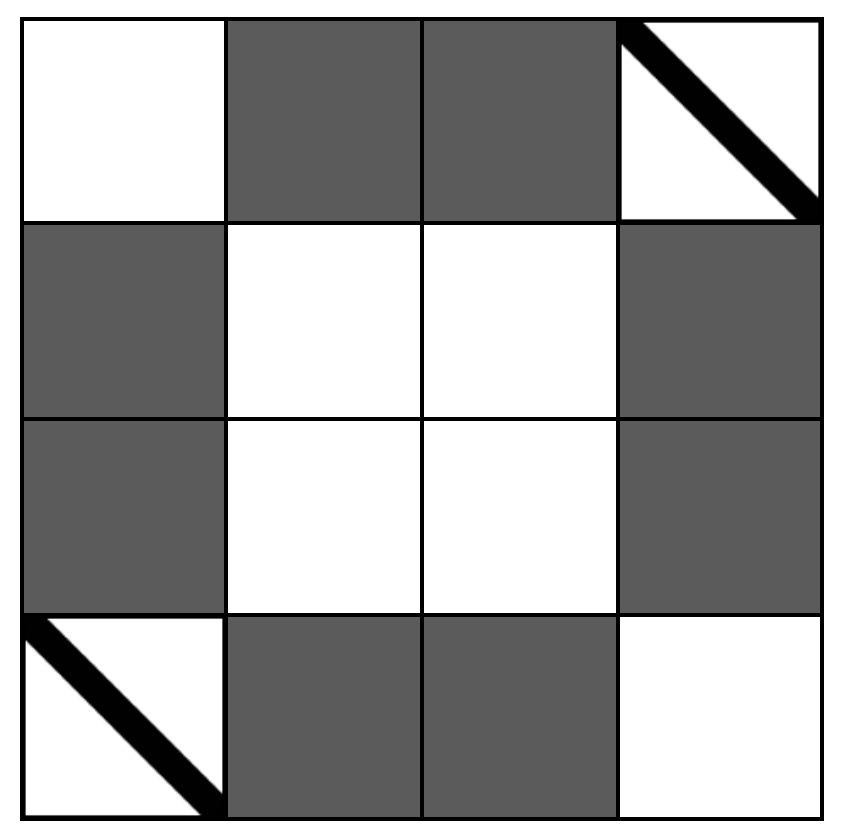}
            & \includegraphics[width=.15\textwidth]{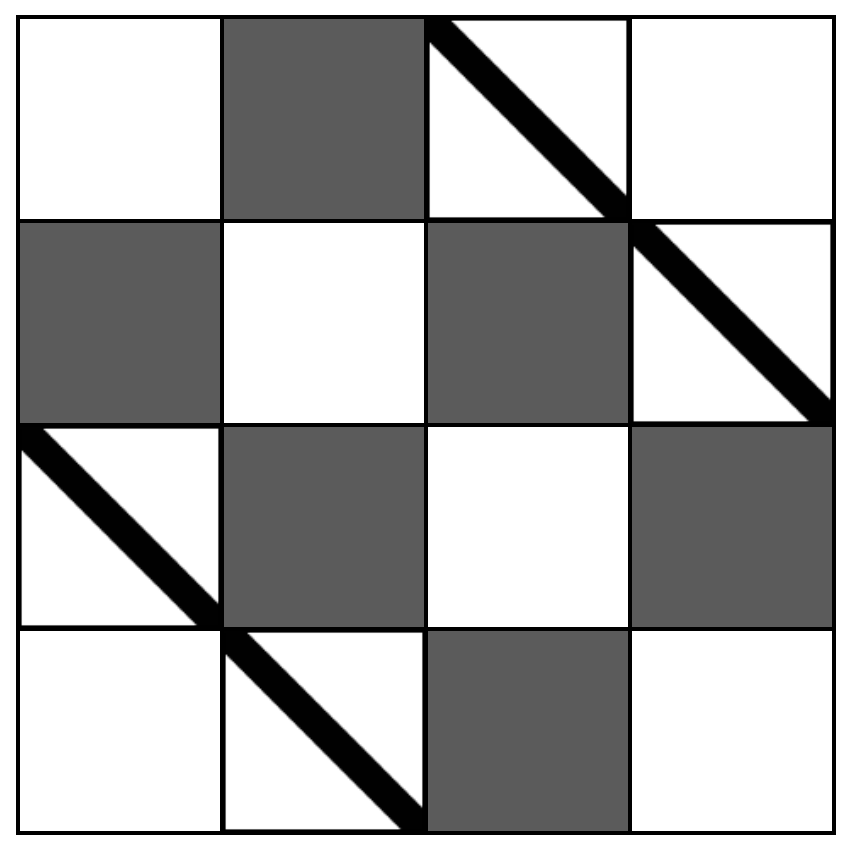}
            & \includegraphics[width=.15\textwidth]{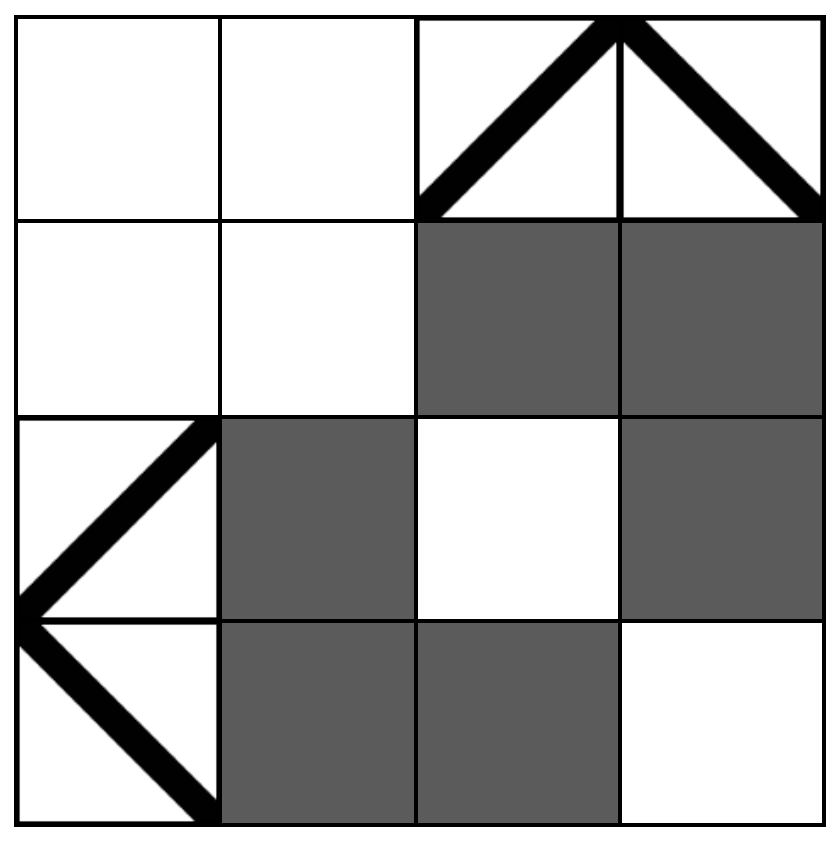} \\
            4
            & \includegraphics[width=.15\textwidth]{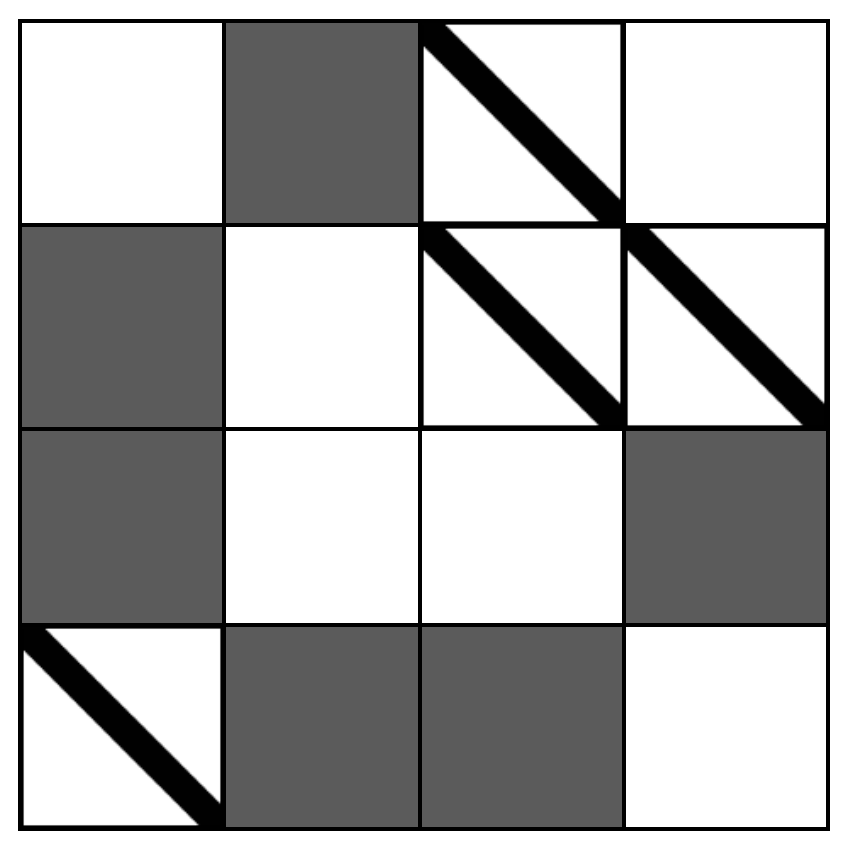}
            & \includegraphics[width=.15\textwidth]{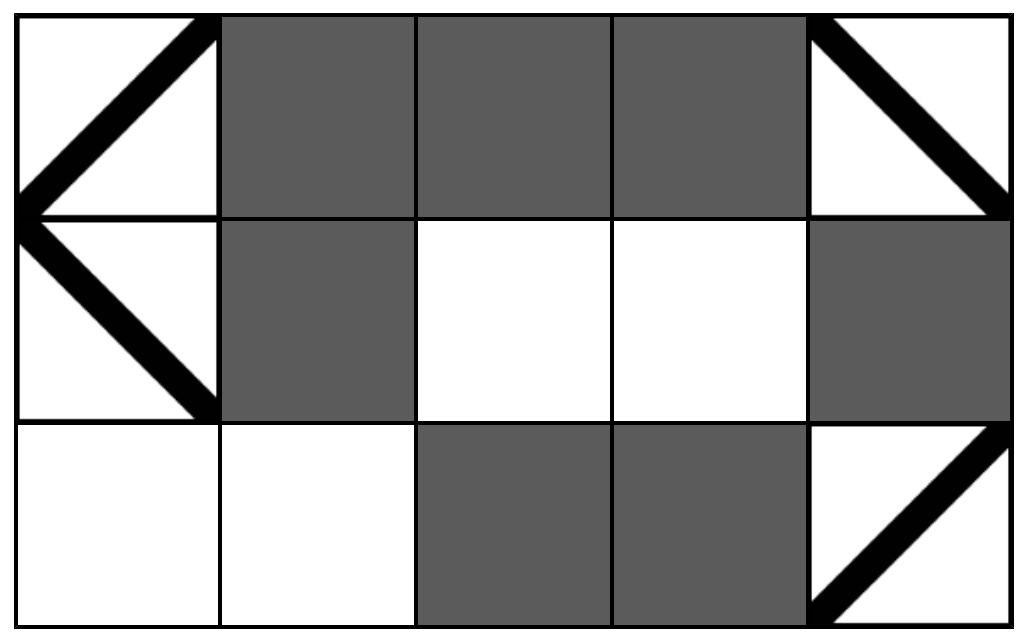}
            & \includegraphics[width=.15\textwidth]{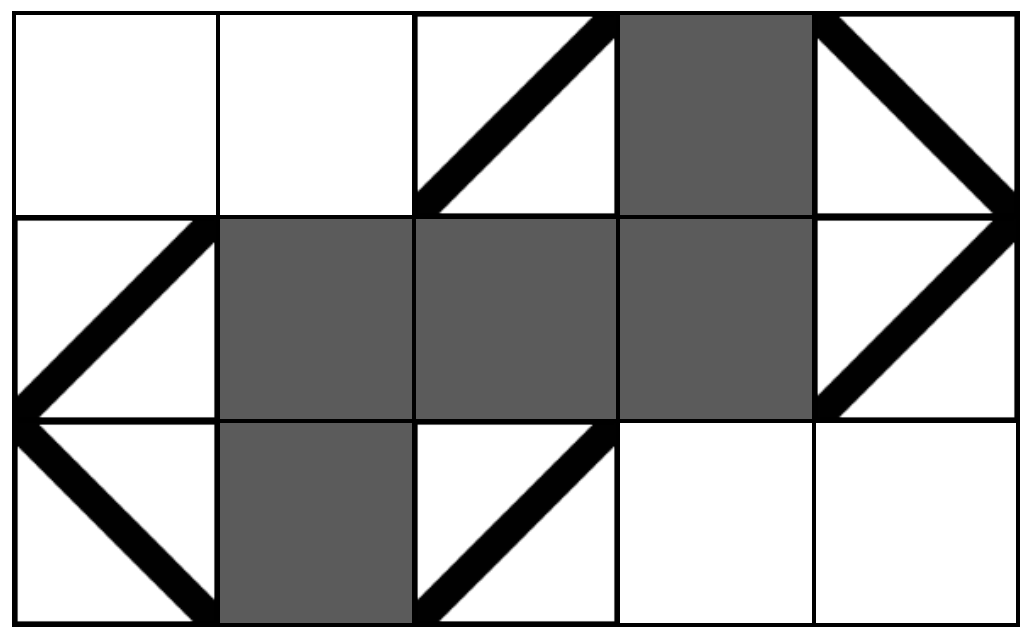}
            & \includegraphics[width=.15\textwidth]{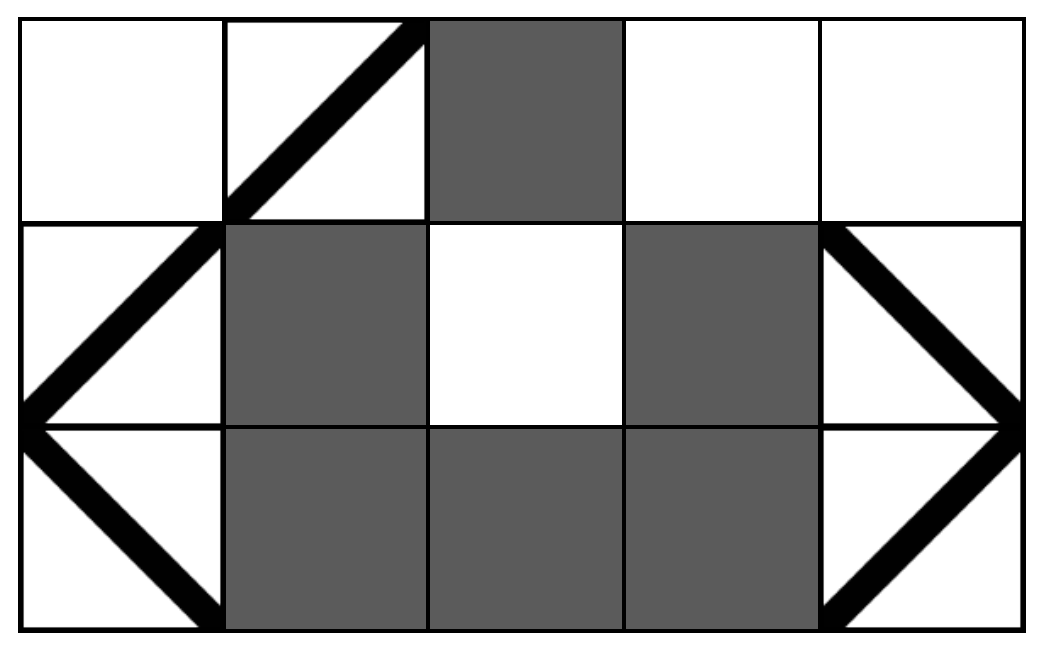}
            & \includegraphics[width=.15\textwidth]{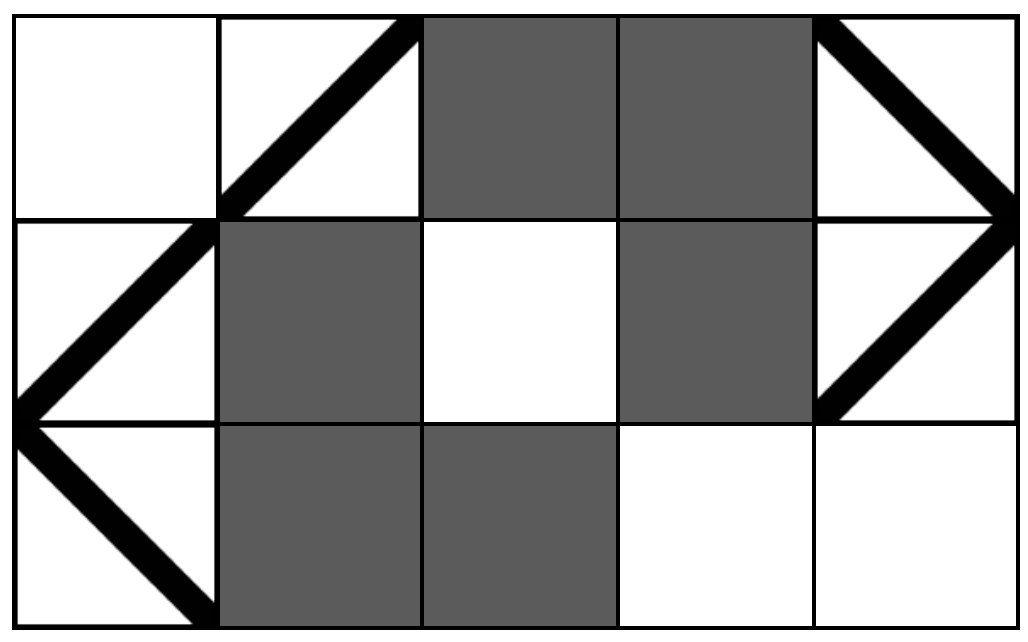} \\
            5
            & \includegraphics[width=.15\textwidth]{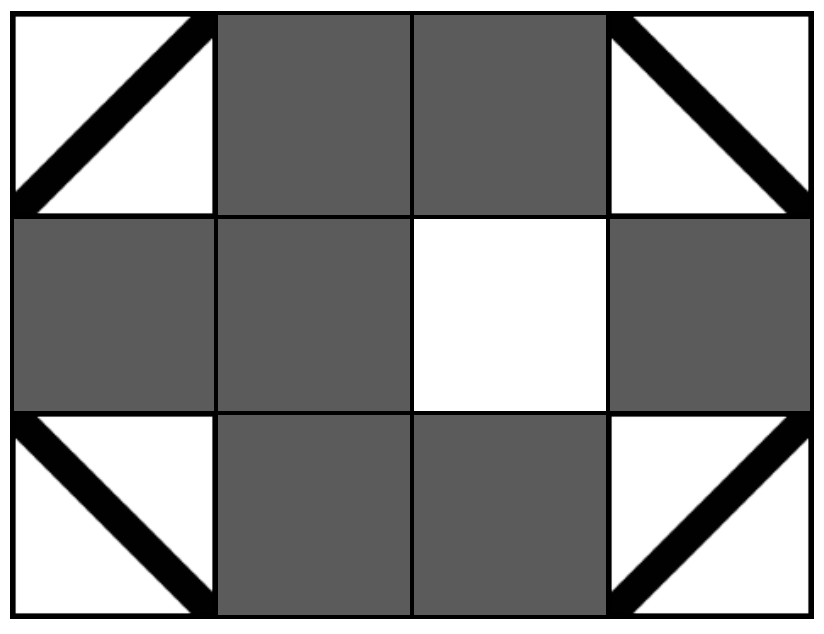}
            & \includegraphics[width=.15\textwidth]{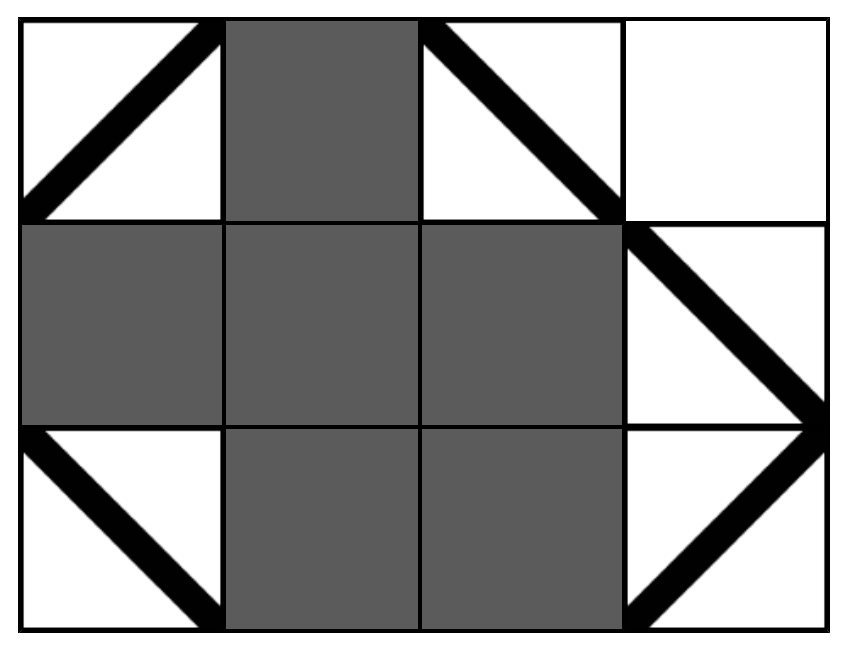}
            & \includegraphics[width=.15\textwidth]{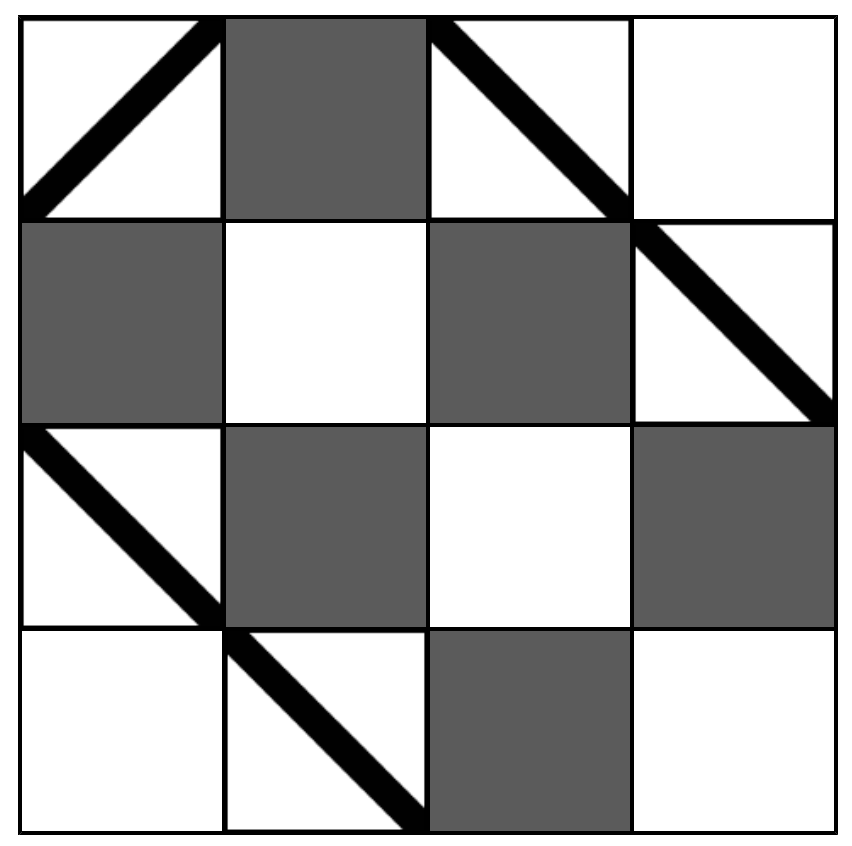}
            & \includegraphics[width=.15\textwidth]{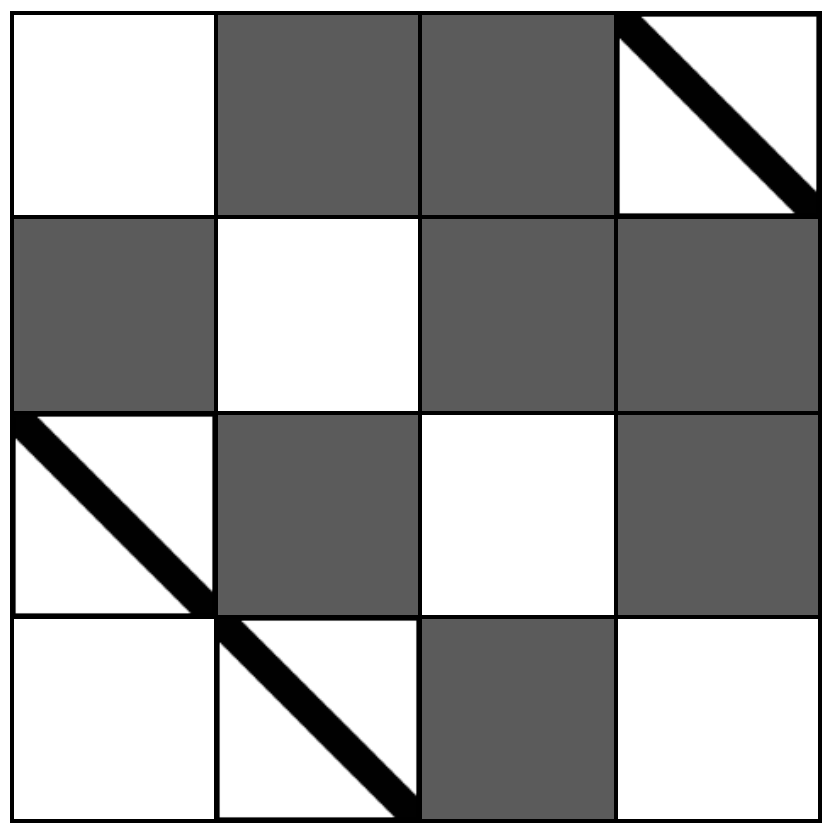}
            & \includegraphics[width=.15\textwidth]{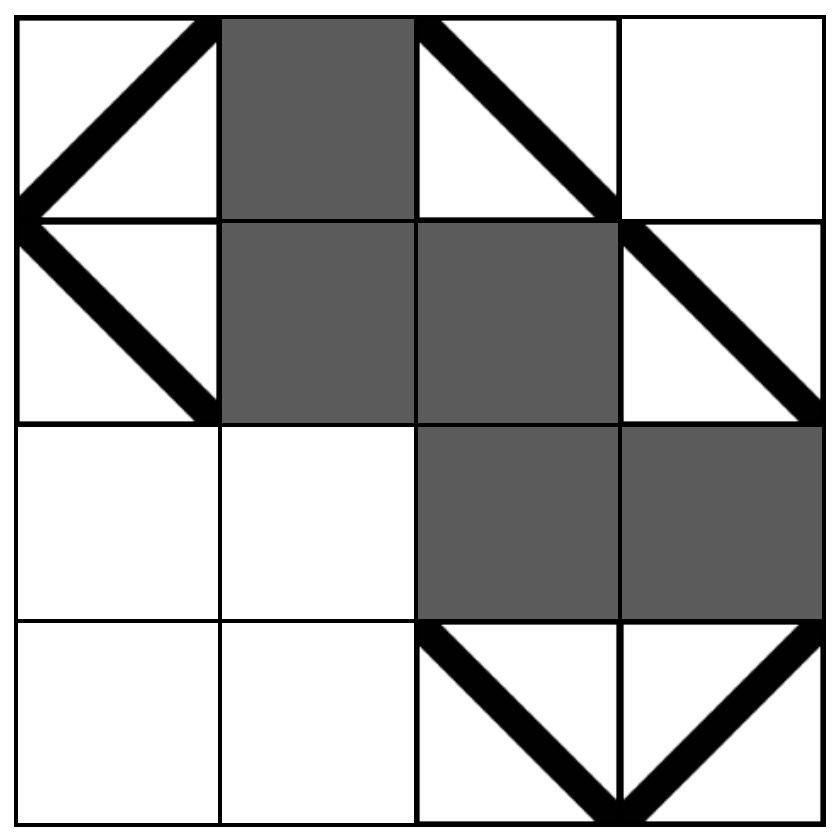} \\
            6
            & \includegraphics[width=.15\textwidth]{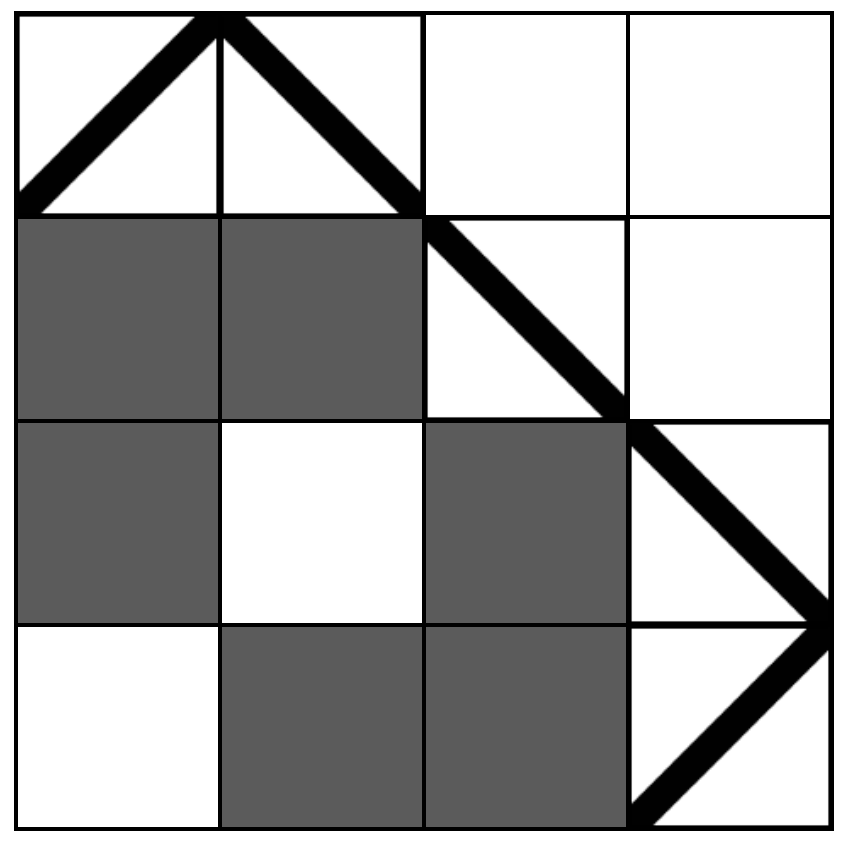}
            & \includegraphics[width=.15\textwidth]{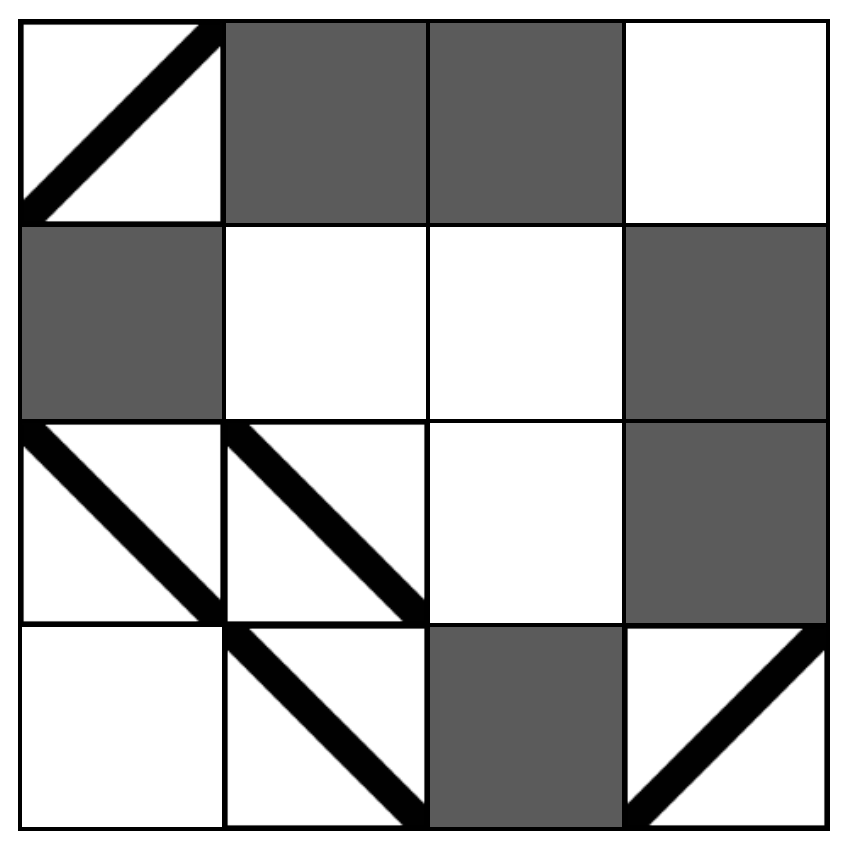}
            & \includegraphics[width=.15\textwidth]{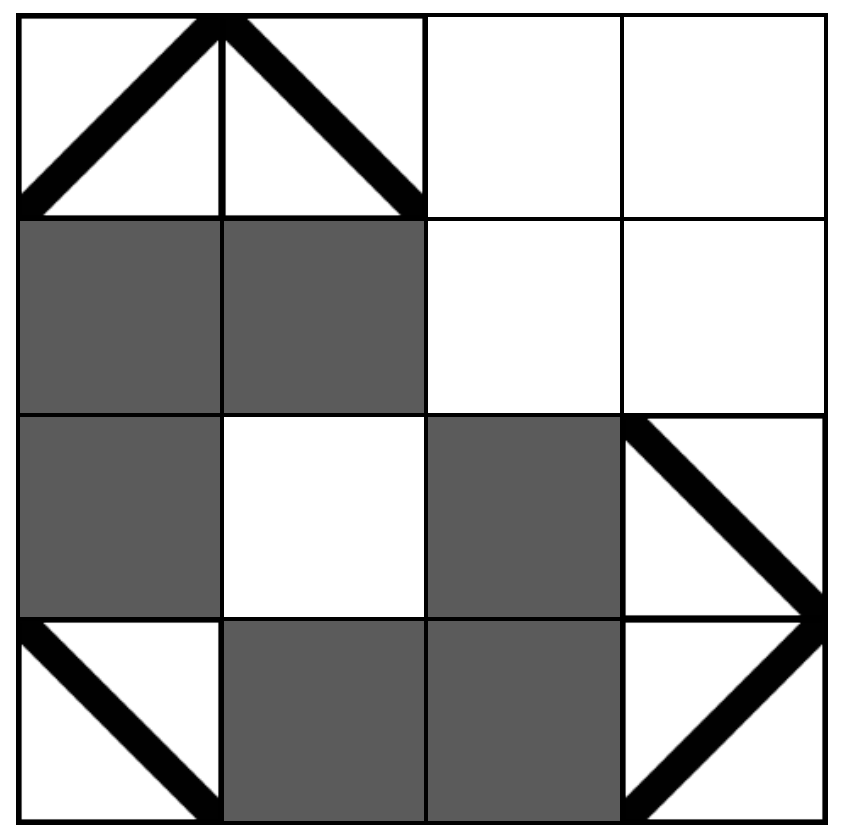}
            & \includegraphics[width=.15\textwidth]{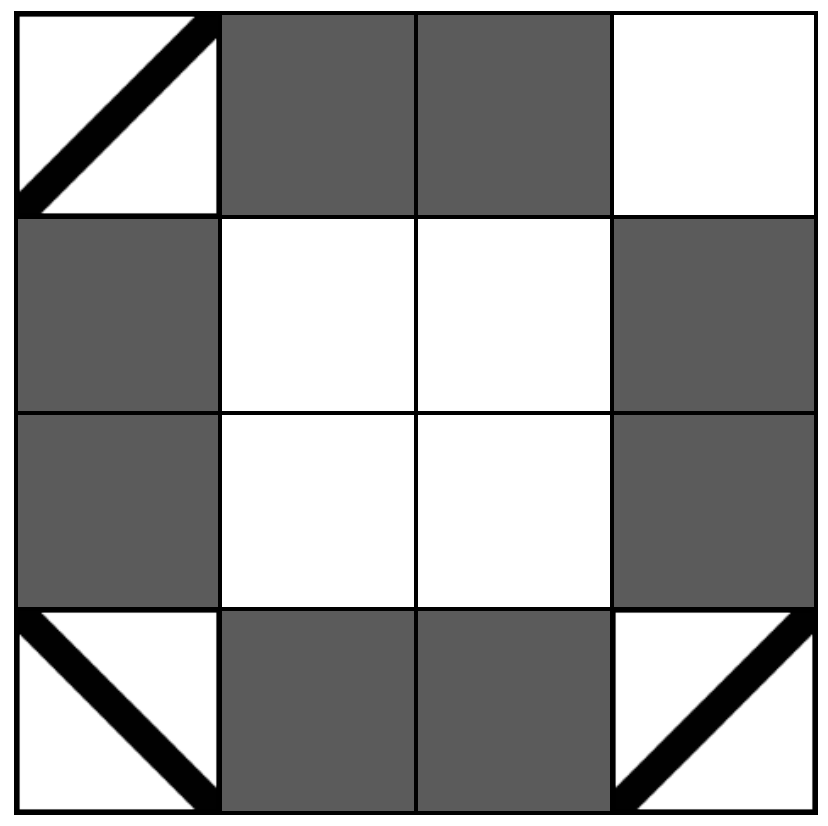}
            & \includegraphics[width=.15\textwidth]{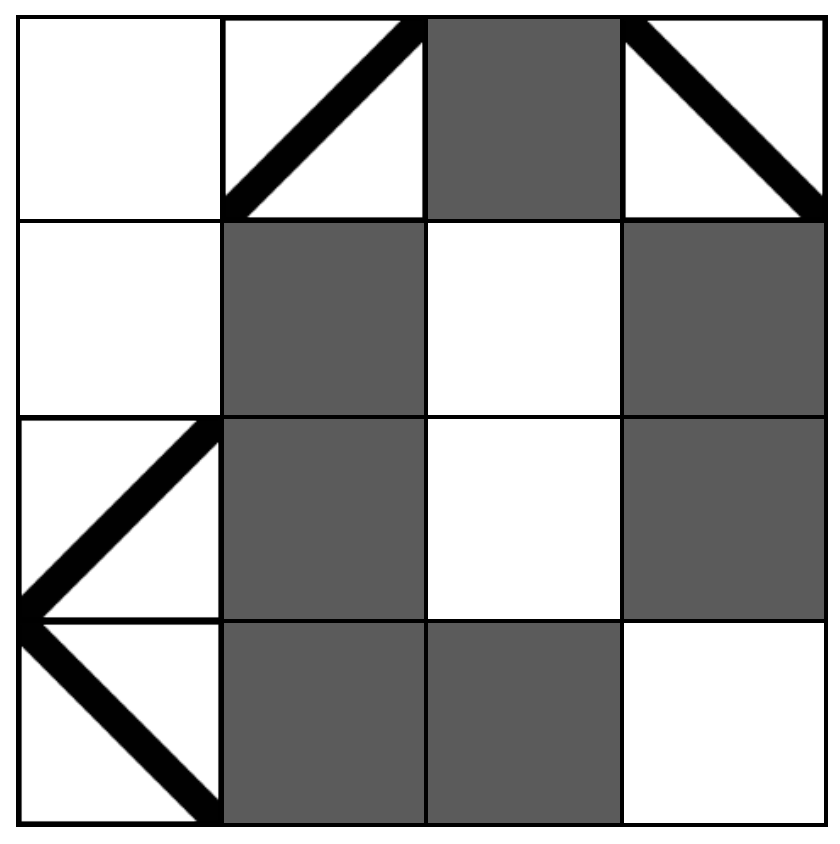} \\
            7
            & \includegraphics[width=.15\textwidth]{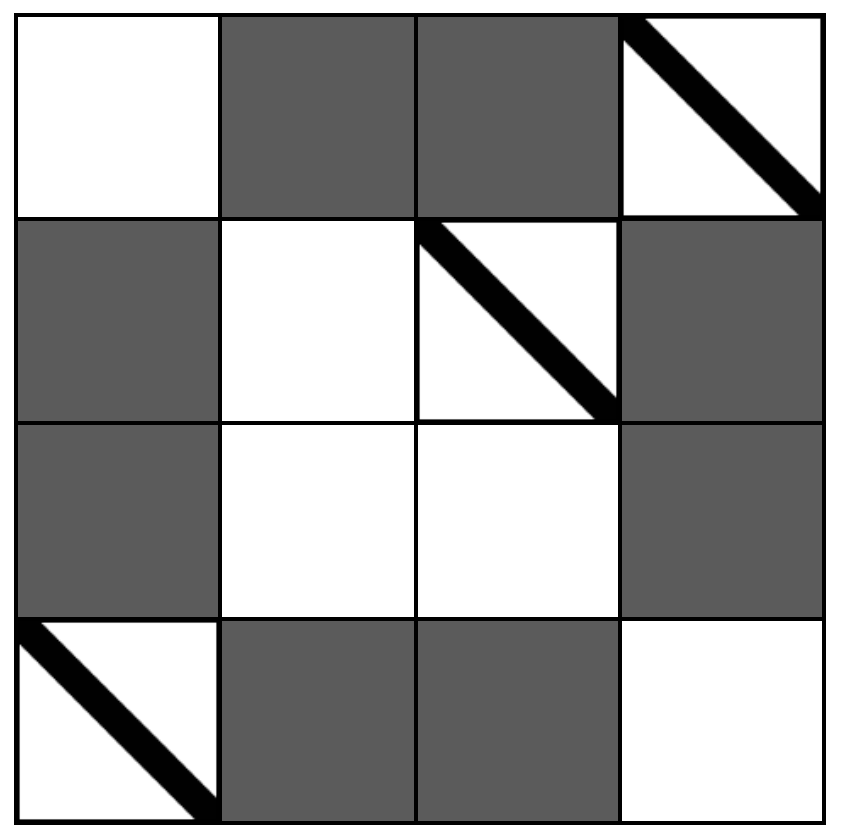}
            & \includegraphics[width=.15\textwidth]{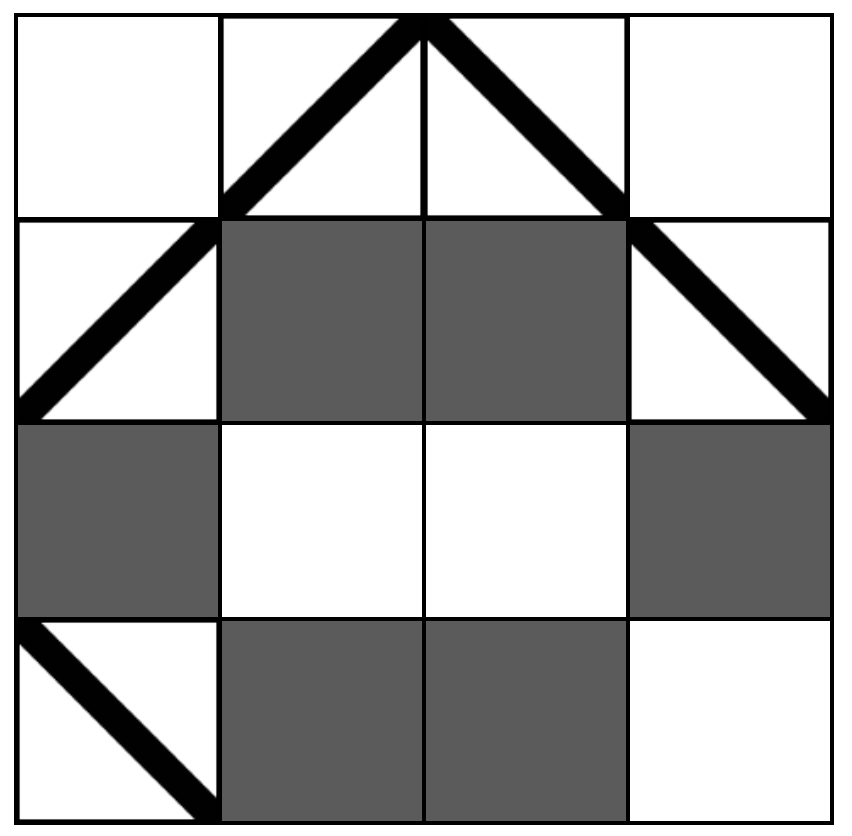}
            & \includegraphics[width=.15\textwidth]{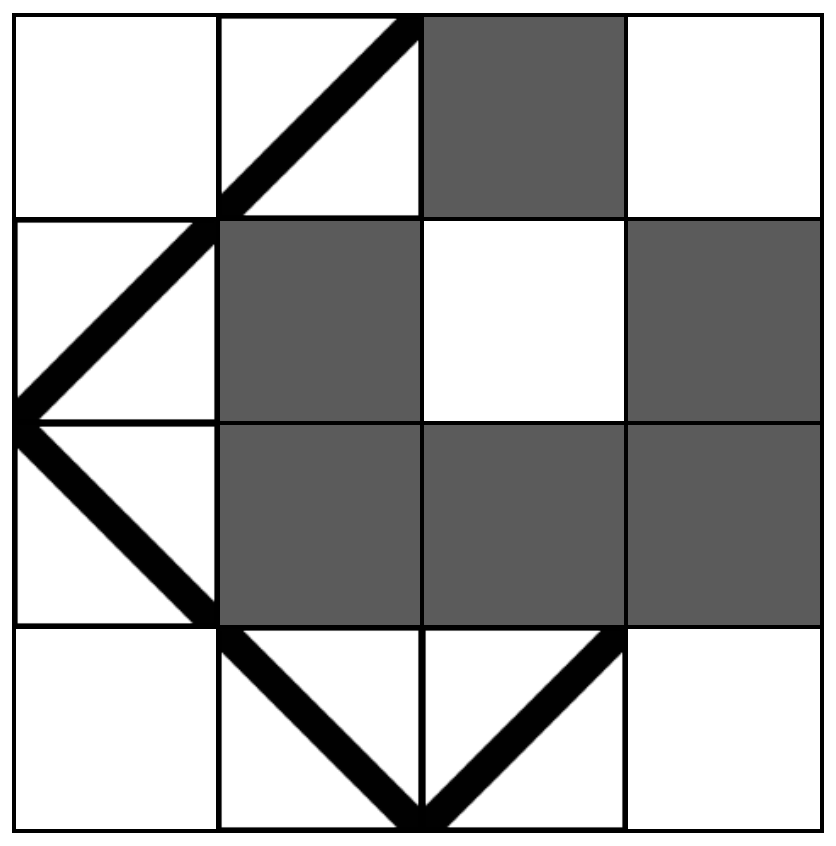}
            & \includegraphics[width=.15\textwidth]{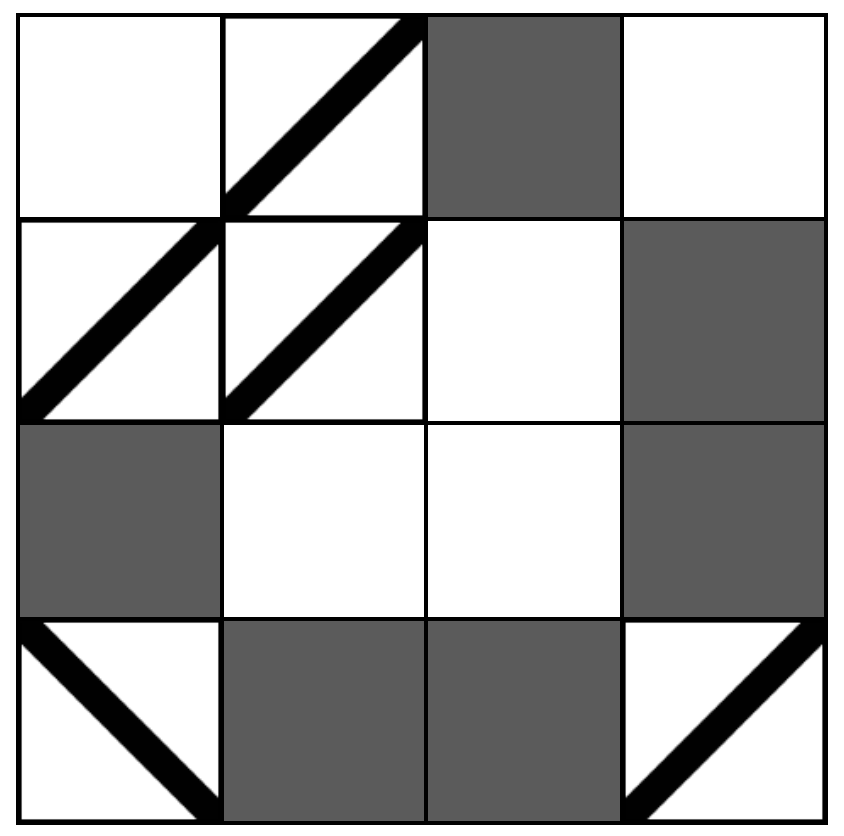}
            & \includegraphics[width=.15\textwidth]{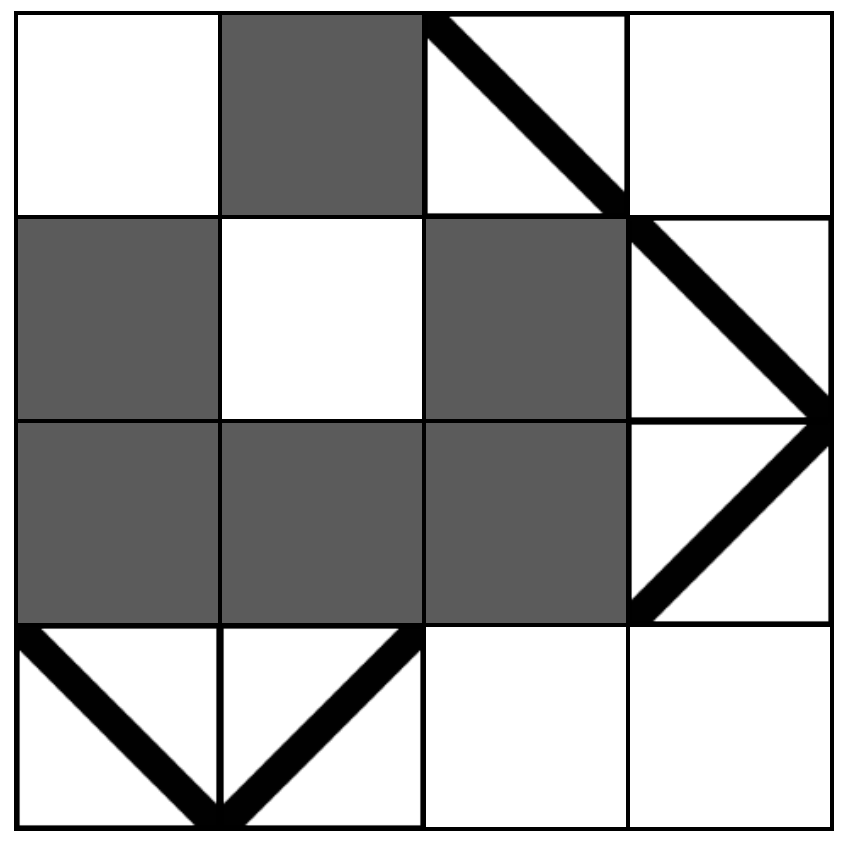} \\
        \end{tabular}
        \caption{Polyominoes from Table \ref{tbl:polyominolist} with corners filled in via Figure \ref{fig:middleist5}.}
        \label{tbl:polyominolistcorners}
    \end{table}
    
    If $L$ has crossing number 3 or less, then it is either the Hopf link (crossing number 2) or the trefoil knot (crossing number 3). Entry 1a has only 2 tiles that were not filled in by Figure \ref{fig:middleist5}, so they must both be crossing tiles. Therefore, $L$ is the Hopf link depicted in Table \ref{tbl:examples}. Since every other polyomino in Table \ref{tbl:polyominolist} contains more than 6 nonempty tiles, we can assume for the rest of the proof that $L$ is not the Hopf link, and hence that any mosaic for $L$ must contain at least 3 crossing tiles.

    By casework, we can rule out entries 1b and 1c. For entry 1d, $L$ is either the trefoil knot or Solomon's knot, as in Table \ref{tbl:examples}. Since the other polyominoes not yet discussed in Table \ref{tbl:polyominolist} contain more than 8 nonempty tiles, we can assume for the rest of the proof that $L$ is not the Hopf link or the trefoil knot or Solomon's knot, and hence that any mosaic for $L$ must contain at least 4 crossing tiles.

    By casework, we can rule out all remaining entries other than 2a, 2b, 3b, 3e, and 5a. For example, examining entry 6a in Table \ref{tbl:polyominolistcorners}, there are only $2$ places where it is possible for crossing tiles to go, less than the required $4$. And in entry 7d, no tile can go in the dark gray slot in the top row without making it apparent that the corner mosaic uses more tiles than necessary.
    
    For entries 2a, 2b, and 3e, $L$ is the connect sum of two Hopf links. One of these Hopf links is shown in Table \ref{tbl:examples}. For entry 3b, $L$ is the cinquefoil knot or the star of David link, as shown in Table \ref{tbl:examples}. For entry 5a, $L$ is the figure-eight knot or the three-twist knot, as shown in Table \ref{tbl:examples}.
\end{proof}

\section*{Acknowledgements}

We thank Boyu Zhang for his advising, Eric Rawdon for his correspondence, and Aaron Heap and his students for Knot Mosaic Maker \cite{maker}, which we used and modified to create many of our figures.

\printbibliography

\end{document}